\documentclass{article}

\title{$q$-congruences, with applications to supercongruences and the cyclic sieving phenomenon}
\author{Ofir Gorodetsky}
\date{}

\usepackage[margin=1in]{geometry}
\usepackage[all]{xy}
\usepackage{amssymb}
\usepackage{amsfonts}
\usepackage{amsthm}
\usepackage{amsmath}
\usepackage{hyperref}
\usepackage{mathtools}
\usepackage{url}

\newtheorem*{thm*}{Theorem}
\newtheorem{thm}{Theorem}[section]

\newtheorem{lem}[thm]{Lemma}  
\newtheorem{cor}[thm]{Corollary}

\newtheorem{proposition}{Proposition}[section]

\newtheorem{remark}[proposition]{Remark}
\newtheorem{example}[proposition]{Example}
\theoremstyle{definition}

\newtheorem{definition}[proposition]{Definition}

\numberwithin{equation}{section}

\begin{document}

\maketitle
	\begin{abstract}
		We establish a supercongruence conjectured by Almkvist and Zudilin, by proving a corresponding $q$-supercongruence. Similar $q$-supercongruences are established for binomial coefficients and the Ap\'{e}ry numbers, by means of a general criterion involving higher derivatives at roots of unity. Our methods lead us to discover new examples of the cyclic sieving phenomenon, involving the $q$-Lucas numbers.
	\end{abstract}

\section{Introduction}
A sequence of integers $\{ a_n \}_{n \ge 1}$ is said to satisfy the Gauss congruences if, for all positive integers $n$,
\begin{equation}\label{neckcond}
\sum_{d \mid n} \mu(d) a_{\frac{n}{d}} \equiv 0 \bmod n,
\end{equation}
where $\mu$ is the usual M\"obius function, defined as
\begin{equation*}
\mu(n) =\begin{cases} (-1)^k & \mbox{if $n=\prod_{i=1}^{k} p_i$, $p_i$ distinct primes,} \\ 0 & \mbox{otherwise.}\end{cases}
\end{equation*}
Gauss proved that $a_n = a^n$ satisfies the Gauss congruences for any prime number $a$. Gauss's result was later extended (independently by various authors) to the following family of examples: $a_n = \mathrm{Tr}(A^n)$ where $A$ can be any square matrix over the integers. See the paper of Zarelua \cite{zarelua2008} for a detailed survey and proofs (cf. Steinlein \cite{steinlein2017} and Corollary~\ref{corexamg} below). 
Other terms for a sequence satisfying the Gauss congruences include `Gauss sequence' \cite{gillespie1989,minton2014}, `generalized Fermat sequence' \cite{du2003}, and `Dold sequence' \cite[Ch.~3.1]{jezierski2006}.
It is known that condition \eqref{neckcond} holds if and only if the following holds for all primes $p$ and positive integers $n,k$:
\begin{equation}\label{primecond}
a_{p^{k}n} \equiv a_{p^{k-1}n} \bmod p^{k},
\end{equation}
see Proposition~\ref{gausscrit}.
For instance, if $a \ge b$ are positive integers, it is known that $a_n = \binom{an}{bn}$ satisfies \eqref{primecond} for all primes $p$ \cite[Ch.~7.1.6]{robert2000}, and so it satisfies the Gauss congruences.

In this paper we introduce the following $q$-analogue of \eqref{neckcond}, which seems to be new. First, for any positive integer $n$, define the following polynomial in variable $q$, which attains the value $n$ at $q=1$:
\begin{equation*}
[n]_q = \frac{q^n-1}{q-1} \in \mathbb{Z}[q].
\end{equation*}
\begin{definition}
A sequence of polynomials $\{ a_n(q)\}_{n \ge 1} \subseteq \mathbb{Z}[q]$ is said to satisfy the `$q$-Gauss congruences' if, for all positive integers $n$,
\begin{equation}\label{qneckcond}
\sum_{d \mid n}\mu(d) a_{n/d}(q^d) \equiv 0 \bmod [n]_q.
\end{equation}
\end{definition}
In \eqref{qneckcond}, the condition $f(q) \equiv 0 \bmod g(q)$ for polynomials $f,g \in \mathbb{Z}[q]$ means $f(q)/g(q) \in \mathbb{Z}[q]$. Since in \eqref{qneckcond} the modulus $g$ is a monic polynomial, Gauss's lemma tells us that the weaker condition $f(q) / g(q) \in \mathbb{Q}[q]$ is equivalent to $f(q)/g(q) \in \mathbb{Z}[q]$. Some simple examples of sequences that satisfy the $q$-Gauss congruences are $a_n(q) =1$ and $a_n(q)=q^n$.

An important consequence of the definition, which is the main motivation behind this paper, is that if $\{ a_n(q) \}_{n \ge 1}$ satisfies the $q$-Gauss congruences, then $\{ a_n(1) \}_{n \ge 1}$ satisfies the Gauss congruences -- this follows by substituting $q=1$ in \eqref{qneckcond}. So one possible way to prove that a sequence satisfies the Gauss congruences is to find a $q$-analogue of it that satisfies the $q$-Gauss congruences. As demonstrated in recent works of Guo and Zudilin \cite{guo2018} and Straub \cite{straub2019}, the approach of establishing congruences via $q$-congruences is fruitful because of additional techniques available in the $q$-setting. In this work we make heavy use of the derivative and its properties.
\begin{remark}
In Lemma~\ref{lem:equivsq} we show that if $\{ a_n(q)\}_{n \ge 1}$ satisfies the $q$-Gauss congruences, then for all primes $p$ and all $n,k \ge 1$
\begin{equation}\label{eq:qanaloguepower}
a_{p^kn}(q) \equiv a_{p^{k-1}n}(q^p) \bmod [p^k]_q.
\end{equation}
There is no implication in the reverse direction.
\end{remark}
\begin{remark}
The special case $k=1$ of \eqref{primecond} is a special case of the Lucas congruences, while the special case $k=1$ of \eqref{eq:qanaloguepower} is a special case of the $q$-Lucas congruences.
\end{remark}
\subsection{Notation}
We use the following notation throughout. For any positive integer $n$, let $\omega_n= e^{\frac{2\pi i}{n}} \in \mathbb{C}$ be a primitive root of unity of order $n$, let $\Phi_n(q) \in \mathbb{Z}[q]$ be the $n$-th cyclotomic polynomial, and set $\mu_n = \{ \omega_n^i : i \in \mathbb{Z} \}$. Given $\omega \in \mu_n$, we write $\mathrm{ord}(\omega)$ for its order. The notation $[u^n]f(u)$, where $f$ is a power series in $u$, means the coefficient of $u^n$ in $f$. We will often write $(a,b)$ instead of $\gcd(a,b)$. For any $n \ge 1$, we set
\begin{equation*}
[n]_q! = \prod_{i=1}^{n} [i]_q,
\end{equation*}
and also $[0]_q!=1$. We define, for all $n \ge k \ge 0$, 
\begin{equation*}
{n \brack k}_q = \frac{[n]_q!}{[k]_q! [n-k]_q!}.
\end{equation*}
The rational functions ${n \brack k}_q$ are in fact polynomials in $\mathbb{Z}[q]$, known as Gaussian binomial coefficients or $q$-binomial coefficients \cite{cohn2004}. Their value at $q=1$ is $\binom{n}{k}$. The $q$-binomial coefficients satisfy the $q$-binomial theorem \cite[Ch.~3.2]{chan2011}:
\begin{equation}\label{qbinom}
\prod_{i=0}^{n-1} (1+tq^i) = \sum_{k=0}^{n} {n \brack k}_{q} t^k q^{\binom{k}{2}}.
\end{equation}
We may also take \eqref{qbinom} to be the definition of ${n \brack k}_{q}$. 
\subsection{First examples}\label{secrel}
A main theme in this work is roots of unity. As we shall show in Corollary~\ref{cor:sp}, a sequence $\{ a_n(q) \}_{n \ge 1} \subseteq \mathbb{Z}[q]$ satisfies the $q$-Gauss congruences if and only if 
\begin{equation}\label{eq:Critanomega}
a_n(\omega) = a_{\frac{n}{\mathrm{ord}(\omega)}}(1)
\end{equation}
for all $n \ge 1$ and $\omega \in \mu_n$. This criterion provides almost immediately the examples below, verified in \S\ref{subsecweakthm}. Let $a \ge b \ge 1$ be integers.
\begin{example}
The sequence $a_n(q)={an \brack bn}_q$ satisfies the $q$-Gauss congruences. It is a $q$-analogue of $a_n(1) = \binom{an}{bn}$. Up to a power of $q$, ${an \brack bn}_q$ is the coefficient of $t^{bn}$ in $\prod_{i=0}^{an-1}(1+tq^i)$, as follows from \eqref{qbinom}.
\end{example}
\begin{example}
The sequence $b_n(q) = {an-1 \brack bn}_q$ satisfies the $q$-Gauss congruences. It is a $q$-analogue of $b_n(1) = \binom{an-1}{bn}$. It is equal to the coefficient of $t^{bn}$ in $\prod_{i=0}^{an-bn-1}(1-tq^i)^{-1}$, as follows from the $q$-binomial series \cite[Ch.~3.2]{chan2011}:
\begin{equation*}
\prod_{i=0}^{n-1} (1-tq^i)^{-1} = \sum_{k \ge 0} {n+k-1 \brack k}_{q} t^k .
\end{equation*}
\end{example}
\begin{example}
The sequence $c_n(q) = [t^{bn}]\prod_{i=0}^{n-1}(1-tq^i)^{a}$ satisfies the $q$-Gauss congruences. It is a $q$-analogue of $c_n(1)=(-1)^{bn}\binom{an}{bn}$. Choosing $a=b=1$, $c_n(q)$ equals $(-1)^{n} q^{\binom{n}{2}}$, which is a $q$-analogue of $(-1)^n$ that satisfies the $q$-Gauss congruences. 
\end{example}
\subsection{Main results}\label{secre2}
The following theorem, proved in \S\ref{subsecweakthm}, provides interesting examples of sequences satisfying the $q$-Gauss congruences. The theorem is established using criterion \eqref{eq:Critanomega}.
\begin{thm}\label{propweakbinom}
The following sequences satisfy the $q$-Gauss congruences.
\begin{enumerate}
\item $d_n(q) = \sum_{i=0}^{\lfloor \frac{n}{2} \rfloor} q^{i(i+b)}{n \brack i}_q {n-i \brack i}_q$, for any integer $b \ge -1$.
\item $e_n(q) = \mathrm{Tr}(A(q^{n-1})A(q^{n-2}) \cdots A(1))$, where
\begin{equation*}
A(x)=\begin{bmatrix}
    1  & x \\
    1 & 0
\end{bmatrix} \in \mathrm{Mat}_{2}(\mathbb{Z})[x].
\end{equation*}
\end{enumerate}
\end{thm}
The sequence $d_n(q)$ is a $q$-analogue of $d_n(1) =\sum_{i=0}^{\lfloor \frac{n}{2} \rfloor} \binom{n}{i}\binom{n-i}{i}$, the central trinomial coefficient, that is, the $n$-th coefficient of $(1+x+x^2)^n$. This $q$-analogue was first introduced by Andrews and Baxter \cite[Eq.~(2.7);\, $A=0$]{andrews1987} in the context of statistical mechanics. More generally, the trinomial coefficient $\binom{n}{a}_2$ is the coefficient of $x^a$ in $(1+x+x^{-1})^n$, and it has the $q$-analogue \cite[Eq.~(2.7)]{andrews1987} 
\begin{equation}\label{eq:deftrinom}
{n; b; q \brack a}_2 = \sum_{i=0}^{ \lfloor \frac{n-a}{2} \rfloor } q^{i(i+b)} {n \brack i}_q {n-i \brack i+a}_q,
\end{equation}
where $b$ is an integer parameter. This $q$-analogue was studied extensively by Andrews \cite{andrews1990,andrews1990q,andrews1994} and Warnaar \cite{warnaar2001, warnaar2003} and we shall return to it in the next theorem. In \S\ref{seccritgauss} we show that $d_n(1)$ satisfies the Gauss congruences, independently of Theorem~\ref{propweakbinom}. A congruence for $d_n(1)$ of different flavor (namely, of Lucas type) was proved by Deutsch and Sagan \cite[Thm.~4.7]{deutsch2006}.

The sequence $e_n(q)$ is a $q$-analogue of $e_n(1) = \mathrm{Tr}( A(1)^n) = L_n$, the Lucas numbers, defined usually as $L_n=F_{n-1}+F_{n+1}$, where $F_n$ are the Fibonacci numbers. Schur \cite{schur1917}  considered the following $q$-analogues of the Fibonacci numbers, in his study of the Rogers-Ramanujan identities:
\begin{equation*}
\begin{split}
F_n(q) &=F_{n-1}(q)+q^{n-2}F_{n-2}(q), \qquad F_0(q)=0, F_1(q)=1,\\
G_n(q)&=G_{n-1}(q)+q^{n-1}G_{n-2}(q), \qquad G_0(q)=0, G_1(q)=1.
\end{split}
\end{equation*}
These $q$-analogues were studied by Andrews \cite{andrews2004}, Carlitz \cite{carlitz1974,carlitz1975}, Cigler  \cite{cigler2003,cigler2004,cigler2016}, Pan \cite{pan2006,pan2013} and others. As can be shown inductively, we have \cite[Eq.~(1.8)]{cigler2016}
\begin{equation*}
A(q^{n-1})A(q^{n-2}) \cdots A(1) = \begin{bmatrix}
    F_{n+1}(q)  & G_n(q) \\
    F_{n}(q) & G_{n-1}(q)
\end{bmatrix},
\end{equation*}
and so
\begin{equation}\label{deffnsum}
e_n(q) = F_{n+1}(q)+G_{n-1}(q)
\end{equation}
for all $n \ge 1$. The $q$-analogue $e_n(q)$ of the Lucas numbers, as defined in \eqref{deffnsum}, was introduced by Pan, who proved that \cite[Thm.~1.1;\, $(\alpha,\beta,\gamma,\delta)=(0,0,1,1)$]{pan2013}
\begin{equation}\label{eq:panen}
e_n(q) \equiv 1 \bmod \Phi_n(q)
\end{equation}
for all $n \ge 1$ (Pan stated his result for $n \ge 3$, but a short calculation shows that it holds for $n=1,2$ as well). We use \eqref{eq:panen} in the proof of Theorem~\ref{propweakbinom}, which can be considered as a generalization of it. Indeed, using criterion \eqref{eq:Critanomega} with $a_n(q)=e_n(q)$ and with primitive roots of unity of order $n$, one recovers \eqref{eq:panen}.

Theorem~\ref{propweakbinom} suggests new examples of the Cyclic Sieving Phenomenon (CSP). We recall the definition of the CSP, which was first defined by Reiner, Stanton and White \cite{reiner2004}. Let $X$ be a finite set, $C$ be a finite cyclic group acting on $X$, and $f(q)$ be a polynomial in $q$ with non-negative integer coefficients. Then the
triple $(X, C, f(q))$ exhibits the CSP if, for all $g \in C$, we have
\begin{equation*}
|X^g| = f(\omega),
\end{equation*}
where $X^g$ is the fixed point set of $g$, and $\omega$ is a root of unity whose order is the same as $g$'s. The simplest interesting example is probably the following. Let $A_{n,k}$ be the set of words (that is, finite sequences) $w$ of length $n$ with $k$ $0$-s and $n-k$ $1$-s. Let $\mathbb{Z}/n\mathbb{Z}$ act on $A_{n,k}$ by rotation. Then $(A_{n,k}, \mathbb{Z}/n\mathbb{Z}, {n \brack k}_q)$ exhibits the CSP. We suggest Sagan's survey \cite{sagan2011} on the topic. 

The behavior of sequences satisfying the $q$-Gauss congruences on roots of unity, described by criterion \eqref{eq:Critanomega}, makes them plausible candidates for the CSP. Indeed, suppose that $\{ a_n(q) \}_{n \ge 1} \subseteq \mathbb{Z}[q]$ satisfies the $q$-Gauss congruences and that $a_n(q)$ has non-negative coefficients for all $n \ge 1$. Suppose further that there are sets $\{ X_n \}_{n \ge 1}$ such that $|X_n| = a_n(1)$, and $\mathbb{Z}/n\mathbb{Z}$ acts on $X_n$ in such a way that $|X_n^i| = |X_{\gcd(n,i)}|$ for all $n \ge 1$, $i \in \mathbb{Z}/n\mathbb{Z}$. Then, by definition, $(X_n, \mathbb{Z}/n\mathbb{Z}, a_n(q))$ exhibits the CSP for all $n \ge 1$. 

In particular, Theorem~\ref{propweakbinom} gives rise to two families exhibiting the CSP. Let $B_{n,0}$ be the set of words of length $n$ on letters $0$, $1$ and $2$, such that the number of $0$-s is equal to the number of $2$-s. Let $C_n$ be the set of words $w$ of length $n$ on letters $0$ and $1$, such that there are no consecutive $1$-s in $w$, not even cyclically ($w_1=w_n=1$ is not allowed). Let $\mathbb{Z}/n\mathbb{Z}$ act by rotation on both $B_{n,0}$ and $C_n$. A short calculation using Theorem~\ref{propweakbinom} shows that $(B_{n,0},\mathbb{Z}/n\mathbb{Z},d_n(q))$ and $(C_n, \mathbb{Z}/n\mathbb{Z}, e_n(q))$ exhibit the CSP for all $n\ge 1$. The following result, proved in \S\ref{seccsp}, shows that much more is true.
\begin{thm}\label{thm:propcsp}\hspace{2em}
\begin{enumerate}
\item Fix integers $n \ge 1$, $|k| \le n$ and $b \ge -1$. Let $B_{n,k}$ be the set of words $w$ of length $n$ on letters $0$, $1$ and $2$, such that the number of $0$-s minus the number of $2$-s is equal to $k$. Let $\mathbb{Z}/n\mathbb{Z}$ act on $B_{n,k}$ by rotation. Then $(B_{n,k},\mathbb{Z}/n\mathbb{Z},{n;b;q \brack k}_2)$ exhibits the CSP, where ${n;b;q \brack k}_2$ is defined in \eqref{eq:deftrinom}.
\item Fix integers $n \ge 1$, $0 \le k \le \frac{n+1}{2}$. Let $C_{n,k}$ be the set of words $w$ of length $n$ with $k$ $1$-s and $n-k$ $0$-s, such that there are no consecutive $1$-s in $w$, not even cyclically ($w_1=w_n=1$ is not allowed). Let $\mathbb{Z}/n\mathbb{Z}$ act on $C_{n,k}$ by rotation. Set $e_{n,k}(q)=[t^k]\mathrm{Tr}(A(q^{n-1},t)A(q^{n-2},t) \cdots A(1,t))$  where
\begin{equation*}
A(x,t)=\begin{bmatrix}
    1  & x \\
    t & 0
\end{bmatrix} \in \mathrm{Mat}_{2}(\mathbb{Z})[x,t].
\end{equation*}
Then $(C_{n,k},\mathbb{Z}/n\mathbb{Z},e_{n,k}(q))$ exhibits the CSP.
\end{enumerate}
\end{thm}
Note that $e_n(q) = \sum_{k} e_{n,k}(q)$. The polynomial $e_{n,k}(q)$ has a closed form. For a word $w$ of length $n$ on letters $0$ and $1$, let
\begin{equation*}
W_1(w) = \sum_{1 \le i \le n: w_i = 1} (n-i), \qquad W_2(w) = \sum_{1 \le i \le n: w_i = 1} i.
\end{equation*}
In \eqref{enkviacnk} we prove that $e_{n,k}(q)$ is equal to $\sum_{w \in C_{n,k}} q^{W_1(w)}$. Since $w\in C_{n,k}$ if and only if the mirror image of $w$ (namely the word $w_{n},w_{n-1},\ldots,w_1$) is in $C_{n,k}$, this shows that $e_{n,k}(q)=g(n,k)/q^k$, where $g(n,k)$ is a polynomial in $q$ given by
\begin{equation*}
g(n,k)(q) = \sum_{w \in C_{n,k}} q^{W_2(w)}.
\end{equation*}
The polynomial $g(n,k)$ was studied by Carlitz \cite{carlitz1974}, who proved that
\begin{equation*}
g(n,k) = q^{k^2}{n-k+1 \brack k}_q - q^{n+(k-1)^2} {n-k-1 \brack k-2}_q.
\end{equation*}
The following theorem concerns supercongruences, an informal notion referring to congruences where the modulus is a surprisingly high power. Before we state our theorem, we need the following definition. 
\begin{definition}\label{def:deforder}
Let $r$ be a positive integer. A sequence $\{ a_n(q) \}_{n \ge 1} \subseteq \mathbb{Z}[q]$ is said to satisfy the `$q$-Gauss congruences of order $r$' if it satisfies the $q$-Gauss congruences, and in addition, for all $n \ge 1$ and $1 \le j \le r-1$, the function
\begin{equation*}
f_{n,j}\colon \mu_n \to \mathbb{C}, \quad \omega \mapsto \omega^j a^{(j)}_n(\omega),
\end{equation*}
depends only on $\mathrm{ord}(\omega)$.  Here $a_n^{(j)}(q)$ is the $j$-th derivative of $a_n(q)$. 
\end{definition}
As we shall show in Theorem~\ref{thmexpansion}, we have the following implication. If $\{ a_n(q) \}_{n \ge 1} \subseteq \mathbb{Z}[q]$ satisfies the $q$-Gauss congruences of order $r$ then for all $n,k \ge 1$ and primes $p \ge r+1$, we have
\begin{equation}\label{eq:Crithigher} 
a_{np^k}(1) \equiv a_{np^{k-1}}(1) \bmod p^{rk}.
 \end{equation}
The following theorem concerns three different sequences. For each sequence, a more detailed theorem is given later in Theorems~\ref{thm:superqbinom}, \ref{thm:aprey}, \ref{thm:zeta}.
\begin{thm}\label{thm:superqexamples}
The following sequences satisfy the $q$-Gauss congruences of order $3$.
\begin{enumerate}
\item $f_n(q) = {an \brack bn}_q$ for $a \ge b \ge 1$.
\item $g_n(q) = \sum_{k=0}^{n} {n \brack k}_q^2 {n+k \brack k}_q^2 q^{f(n,k)}$, where $f(x,y)=y^2+Axy+Bx^2 \in \mathbb{Z}[x,y]$ is a polynomial satisfying $f(n,k) \ge 0$ for all $n,k \ge 0$.
\item $h_n(q) = \sum_{\substack{0 \le \ell \le k \le n\\n\le k+\ell}} {n \brack k}_q^2 {n \brack \ell}_q {k \brack \ell}_q {k+\ell \brack n}_q q^{f(n,k,\ell)}$, where  $f(x,y,z)=y^2+z^2+Ax^2+Bxy+Cxz+Dyz \in \mathbb{Z}[x,y,z]$ is a polynomial satisfying $f(n,k,\ell) \ge 0$ for all $n,k,\ell \ge 0$.
\end{enumerate}
\end{thm}
From Theorem~\ref{thm:superqexamples} and \eqref{eq:Crithigher} we have the following corollary.
\begin{cor}\label{cor:superexamples}
Let $p \ge 5$ be a prime. Then, in the notation of Theorem~\ref{thm:superqexamples}, for all $n,k\ge 1$ we have
\begin{align}
\label{eq:conseqbinomq}
\binom{ap^k}{bp^k} &\equiv \binom{ap^{k-1}}{bp^{k-1}} \bmod p^{3k},\\
g_{np^k}(1) &\equiv g_{np^{k-1}}(1) \bmod p^{3k},\label{eq:conseqaperyq}\\
\label{eq:conseqzetaq} h_{np^k}(1) & \equiv h_{np^{k-1}}(1) \bmod p^{3k}. 
\end{align}
\end{cor}
The supercongruence \eqref{eq:conseqbinomq} is attributed to Ljunggren and Jacobsthal \cite{brun1949}. The sequence $g_n(q)$ is a $q$-analogue of 
\begin{equation*}
g_n(1) = \sum_{k=0}^{n} \binom{n}{k}^2 \binom{n+k}{k}^2,
\end{equation*}
a sequence of integers named ``Ap{\'e}ry numbers" after Roger Ap{\'e}ry, who introduced them in his proof of the irrationality of $\zeta(3)$ \cite{apery1979},  where $\zeta$ is the Riemann zeta function. The supercongruence \eqref{eq:conseqaperyq} was proved by Beukers \cite{beukers1985} and Coster  \cite{coster1988}. Before \eqref{eq:conseqaperyq} was proved, the special case $k=1$ was conjectured by Chowla, Cowles and Cowles \cite{chowla1980} and proved by Gessel \cite{gessel1982} and Mimura \cite{mimura1983}. Ap{\'e}ry has shown that the sequence $\{g_n(1)\}_{n\ge 1}$ satisfies the recurrence relation
\begin{equation*}
(n+1)^3 u_{n+1} - (2n+1)(17n^2+17n+5)u_n + n^3 u_{n-1}=0, \qquad (u_0=0,u_1=5),
\end{equation*}
which implies that the generating function $F(x)= \sum_{n \ge 1} g_n(1) x^n$ satisfies a third-order differential equation. The function $F(x)$ also enjoys a `modular parameterization', see \cite{beukers1987}. The supercongruence \eqref{eq:conseqzetaq} is new. The sequence $h_n(q)$ is a $q$-analogue of 
\begin{equation*}
h_n(1) = \sum_{k,\ell} \binom{n}{k}^2 \binom{n}{\ell} \binom{k}{\ell} \binom{k+\ell}{n},
\end{equation*}
a sequence of integers introduced originally by Almkvist and Zudilin in their study of Calabi-Yau differential equations \cite{almkvist2006}. The sequence is denoted there by the letter $\zeta$, and this is the way it is referred to in the literature \cite{almkvist2011, osburn2016, malik2016}. Almkvist and Zudilin found the sequence by searching, empirically, for \emph{integer} sequences satisfying the recurrence
\begin{equation*}
(n + 1)^3 u_{n+1} - (2n+1)(an^2+an+b)u_n +cn^3 u_{n-1} =0
\end{equation*}
for some $a,b,c \in \mathbb{Z}$ (see Zagier \cite{zagier2009} and Cooper \cite{cooper2012} for related searches). They found 5 solutions apart from the Ap{\'e}ry numbers (cf. \cite[Eq.~(4.12)]{almkvist2011}, \cite[Table~2]{osburn2016}, \cite[Table~2]{malik2016}), which are often referred to as Ap{\'e}ry-like sequences. One of them is $h_n(1)$, corresponding to $(a,b,c)=(9,3,-27)$. These 5 sequences are conjectured to satisfy the supercongruences
\begin{equation}\label{eq:conjsuper}
u_{np^k} \equiv u_{np^{k-1}} \bmod p^{3k}.
\end{equation}
for all $n,k \ge 1$ and primes $p \ge 5$. For three out of these 5 sequences, \eqref{eq:conjsuper} was proved by Osburn and Sahu \cite{osburn2013} and Osburn, Sahu and Straub \cite{osburn2016}. For another one of these sequences, \eqref{eq:conjsuper} was proved in the case $k=1$ by Amdeberhan and Tauraso \cite{amdeberhan2016}. Malik and Straub \cite[Thm.~3.1]{malik2016} have 
proved that all 5 Ap\'{e}ry-like sequences satisfy Lucas-type congruences.

Theorem~\ref{thmexpansion} below gives us much more than integer congruences. It uses the fact that $a_n(q)$ satisfies the $q$-Gauss congruences of order $r$ to construct an explicit formula for the remainder of $a_{np^k}(q)$ upon division by $[p^k]_q^{r}$ ($p \ge r+1$), let us denote it by $r(q)$. The supercongruence  \eqref{eq:Crithigher} is deduced by substituting $q=1$ in the $q$-supercongruence $a_{np^k}(q) \equiv r(q) \bmod [p^k]_q^{r}$. The details are provided in the next section.
\begin{remark}
Theorem~\ref{thmexpansion} determines not only the remainder of $a_{np^k}(q)$ modulo $[p^k]_q^{r}$ for primes $p \ge r+1$, but actually the remainder of $a_{nm}(q)$ modulo $[m]_q^{r}$ for any $m \ge 1$ which is not divisible by primes less than $r+1$.
\end{remark}
\section{Methods}
Most of the proofs of the results in \S\ref{secrel}--\ref{secre2} follow from Corollaries~\ref{cor:sp} and \ref{corexpansion}, which themselves follow from more general results that we discuss here. To state these results, we introduce some new notions.
\subsection{Gauss congruences with respect to a set of primes}
Let $\mathbb{P}$ denote the set of primes. Given $S \subseteq \mathbb{P}$, we denote by $\mathbb{N}_{S}$ the set of positive integers divisible only by primes from $S$. 
\begin{definition}
Let $S \subseteq \mathbb{P}$. A sequence of integers $\{ a_n\}_{n \ge 1}$ is said to satisfy the `Gauss congruences with respect to $S$' if, for all $n \in \mathbb{N}_S$ and $m \ge 1$,
\begin{equation}\label{weakneckcond}
\sum_{d \mid n} \mu(d) a_{\frac{nm}{d}} \equiv 0 \bmod n.
\end{equation}
A sequence of polynomials $\{a_n(q)\}_{n \ge 1} \subseteq \mathbb{Z}[q]$ is said to satisfy the `$q$-Gauss congruences with respect to $S$' if, for all $n \in \mathbb{N}_S$ and $m \ge 1$,
\begin{equation}\label{qweakneckcond}
\sum_{d \mid n} \mu(d) a_{\frac{nm}{d}}(q^d) \equiv 0 \bmod [n]_q.
\end{equation}
\end{definition}
We see that if $\{ a_n(q) \}_{n \ge 1}$ satisfies the $q$-Gauss congruences with respect to $S$, then $\{a_n(1)\}_{n \ge 1}$ satisfies the Gauss congruences with respect to $S$. In the special case $S=\{p \}$, \eqref{weakneckcond} becomes \eqref{primecond} with the prime $p$ fixed, and \eqref{qweakneckcond} becomes \eqref{eq:qanaloguepower} with the prime $p$ fixed. In \S\ref{seccritgauss}, we prove the following.
\begin{lem}\label{lem:equivs}\hspace{2em}
\begin{enumerate}
\item A sequence $\{a_n\}_{n \ge 1} \subseteq \mathbb{Z}$ satisfies the Gauss congruences with respect to $S \subseteq \mathbb{P}$ if and only if \eqref{primecond} holds for all $p \in S$.
\item A sequence $\{ a_n\}_{n \ge 1} \subseteq \mathbb{Z}$ satisfies the Gauss congruences with respect to $\mathbb{P}$ if and only if it satisfies the Gauss congruences.
\end{enumerate}
\end{lem}
We have a partial $q$-analogue of Lemma~\ref{lem:equivs}, proved in \S\ref{secchar}.
\begin{lem}\label{lem:equivsq}\hspace{2em}
\begin{enumerate}
\item If a sequence $\{a_n(q)\}_{n \ge 1} \subseteq \mathbb{Z}[q]$ satisfies the $q$-Gauss congruences with respect to $S \subseteq \mathbb{P}$ then \eqref{eq:qanaloguepower} holds for all $p \in S$ and all $n,k \ge 1$.
\item A sequence $\{ a_n(q)\}_{n \ge 1} \subseteq \mathbb{Z}[q]$ satisfies the $q$-Gauss congruences with respect to $\mathbb{P}$ if and only if it satisfies the $q$-Gauss congruences.
\end{enumerate}
\end{lem}
In view of the second part of Lemma~\ref{lem:equivsq}, whenever we prove a theorem on sequences satisfying the $q$-Gauss congruences with respect to an arbitrary $S\subseteq \mathbb{P}$, we also obtain, in the special case $S=\mathbb{P}$, a result on sequences satisfying the $q$-Gauss congruences.
\subsection{General results}
 In \S\ref{secchar} we prove the following characterization.
\begin{proposition}\label{propcharweak}
Let $S \subseteq \mathbb{P}$ and $\{a_n(q)\}_{n \ge 1} \subseteq \mathbb{Z}[q]$. The following are equivalent.
\begin{enumerate}
\item $\{a_n(q)\}_{n \ge 1}$ satisfies the $q$-Gauss congruences with respect to $S$.
\item For all $m\ge 1$ and $n\in \mathbb{N}_S$, and every $\omega \in \mu_n$, we have \begin{equation}\label{charweak}
a_{nm}(\omega) =a_{\frac{nm}{\mathrm{ord}(\omega)}}(1).
\end{equation}
\item For all $m\ge 1$ and $n\in \mathbb{N}_S$, and every $d \mid n$, we have
\begin{equation*}
a_{nm}(q)  \equiv a_{\frac{nm}{d}}(1) \bmod \Phi_{d}(q).
\end{equation*}
\end{enumerate}
\end{proposition}
In particular, we have the following immediate corollary of Lemma~\ref{lem:equivsq} and Proposition~\ref{propcharweak}, obtained by taking $S=\mathbb{P}$ in Proposition~\ref{propcharweak}.
\begin{cor}\label{cor:sp}
The sequence $\{a_n(q)\}_{n \ge 1} \subseteq \mathbb{Z}[q]$ satisfies the $q$-Gauss congruences if and only if, for all $n,i \ge 1$, we have
\begin{equation*}
a_n(\omega_n^i) = a_{(n,i)}(1),
\end{equation*}
which holds if and only if, for all $n \ge 1$ and every $d \mid n$,  we have
\begin{equation*}
a_n(q) \equiv a_{n/d}(1) \bmod \Phi_d(q).
\end{equation*}
\end{cor}
\begin{example}
Let $a$ be a positive integer. Pan \cite{pan2008} studied $a_n(q)=\prod_{i=1}^{n} [a]_{q^i}$, a $q$-analogue of $a_n(1)=a^n$. He proved that
\begin{equation*}
\sum_{d \mid n}\mu(d) a_{n/d}(q^d) \equiv 0 \bmod [n]_{q^{\gcd(n,a)}}
\end{equation*}
for all $n \ge 1$. Using Proposition~\ref{propcharweak}, it can be shown quickly that $\{ a_n(q) \}_{n \ge 1}$ satisfies the $q$-Gauss congruences with respect to $S=\mathbb{P} \setminus \{ p \in \mathbb{P}: p \mid a \}$. Indeed, for all $m  \ge 1$, $n \in \mathbb{N}_S$ and $\omega \in \mu_n$, we have
\begin{equation*}
\begin{split}
a_{nm}(\omega) &= \prod_{i=1}^{nm} [a]_{\omega^i} = \left(\prod_{i=1}^{\mathrm{ord}(\omega)} [a]_{\omega^i} \right)^{nm/\mathrm{ord}(\omega)}\\
&= \left( a \prod_{i=1}^{\mathrm{ord}(\omega)-1} \frac{\omega^{ia}-1}{\omega^i-1}\right)^{nm/\mathrm{ord}(\omega)} = a^{nm/\mathrm{ord}(\omega)} = a_{nm/\mathrm{ord}(\omega)}(1).
\end{split}
\end{equation*} 
\end{example}
Our next proposition shows that if $\{ a_n(q)\}_{n \ge 1}$ satisfies the $q$-Gauss congruences with respect to $S$, then we can determine the remainder of $a_{nm}(q)$ upon division by $[n]_q$ as long as $n \in \mathbb{N}_S$. It will be convenient to introduce the following polynomials.
\begin{definition}\label{def:ggn}
For any $n \ge 1$, let $D_n=\{ d : d \text{ a divisior of }n, \, d \neq n\}$ be the set of proper divisors of $n$. For a function $g\colon D_n \to \mathbb{C}$, let
\begin{equation*}
G_{g,n}(q) = \sum_{d \in D_n} \frac{[n]_q}{[n/d]_q} \frac{\sum_{e \mid d} \mu(\frac{d}{e}) g(e)}{d} \in \mathbb{C}[q].
\end{equation*}
\end{definition}
The main property of $G_{g,n}$ is that the value of $G_{g,n}(\omega)$ for $\omega \in \mu_n$ depends only on the order of $\omega$. In \S\ref{secchar} we establish the following formula. 
\begin{proposition}\label{proprem}
Assume that the sequence $\{ a_n(q) \}_{n \ge 1} \subseteq \mathbb{Z}[q]$ satisfies the $q$-Gauss congruences with respect to $S \subseteq \mathbb{P}$. Let $n,m$ be positive integers with $n \in \mathbb{N}_{S}$. Then the remainder of $a_{nm}(q)$ upon division by $[n]_q$ is \begin{equation*}
a_{nm}(q) \equiv G_{g,n}(q)  \bmod [n]_q,
\end{equation*}
where
\begin{equation*}
g\colon D_n \to \mathbb{C}, \quad g(d) = a_{dm}(1).
\end{equation*}
\end{proposition}
Our next theorem concerns sequences satisfying the $q$-Gauss congruences of order $r$ with respect to $S \subseteq \mathbb{P}$. This notion generalizes Definition \ref{def:deforder}, which corresponds to the special case $S= \mathbb{P}$. 
\begin{definition}
Let $S \subseteq \mathbb{P}$ and $r$ a positive integer. A sequence $\{ a_n(q) \}_{n \ge 1} \subseteq \mathbb{Z}[q]$ is said to satisfy the `$q$-Gauss congruences of order $r$ with respect to $S$' if it satisfies the $q$-Gauss congruences with respect to $S$, and in addition, for all $ m \ge 1$, $n \in \mathbb{N}_S$ and $1 \le j \le r-1$, the function
\begin{equation*}
f_{n,m,j}\colon \mu_n \to \mathbb{C}, \quad \omega \mapsto \omega^{j} a_{nm}^{(j)}(\omega)
\end{equation*}
depends only on $\mathrm{ord}(\omega)$.
\end{definition}
The next theorem shows that if $\{ a_n(q)\}_{n \ge 1}$ satisfies the $q$-Gauss congruences of order $r$ with respect to $S$, then we can determine the remainder of $a_{nm}(q)$ upon division by $[n]_q^{r}$ as long as $n \in \mathbb{N}_S$ and $n$ is not divisible by primes less than $r+1$. Our theorem is most easily stated using the notion of $[n]_q$-digits of a polynomial. 
\begin{definition}
For any polynomial $F(q) \in \mathbb{C}[q]$ and any integer $n>1$, we may expand $F$ in base $[n]_q$, that is, we may write
\begin{equation*}
F(q) = \sum_{i=0}^{\lfloor \deg F / (n-1) \rfloor} f_i(q) [n]^i_{q},
\end{equation*}
with $\deg f_i(q) < n-1$. For every $i$, $f_i(q)$ is unique and we refer to it as the \emph{$i$-th $[n]_q$-digit of $F$}. We define $f_i =0$ for $i> \lfloor \deg F / (n-1) \rfloor$.
\end{definition}
If $F(q) \in \mathbb{Z}[q]$, then since $[n]_q$ is monic, the polynomials $\{ f_i(q) \}_{i=0}^{\lfloor \deg F / (n-1)\rfloor}$ \emph{must} have integer coefficients. In particular, for all $k \ge 0$,
\begin{equation*}
F(1) \equiv \sum_{i=0}^{k-1} f_i(1) n^i \bmod n^k.
\end{equation*}
We see that determining $[n]_q$-digits of $F$ gives us information on $F(1)$ modulo higher powers of $n$. We may now state the theorem, which is an extension of Proposition~\ref{proprem}. 
\begin{thm}\label{thmexpansion}
Assume that $\{ a_n(q) \}_{n \ge 1} \subseteq \mathbb{Z}[q]$ satisfies the $q$-Gauss congruences of order $r$ with respect to $S \subseteq \mathbb{P}$, for some $r \ge 0$. Let $n \ge 2$, $m\ge 1$ be integers with $n \in \mathbb{N}_S$. Let $p$ be the smallest prime divisor of $n$. For any $0 \le i \le r-1$, define the function
\begin{equation*}
g_{i}\colon D_n \to \mathbb{C}, \quad g_{i}(d) = (\omega_n^d)^i a_{mn}^{(i)}(\omega_n^d).
\end{equation*}
Then the first $1+\min\{ p-2,r-1\}$ $[n]_q$-digits of $a_{mn}(q)$ are given recursively by
\begin{equation}\label{eq:digitsthmexp}
f_i(q) = \frac{1}{i!n^i} \left( (q-1)^{i} G_{g_{i},n}(q) - \sum_{m_1=0}^{i-1} \sum_{m_2=m_1}^{i} \binom{i}{m_2} f_{m_1}^{(i-m_2)}(q)R_{n,m_1,m_2}(q)(q-1)^{i-m_2}q^{i-m_2}\right),
\end{equation}
for $0 \le i \le \min\{p-2, r-1\}$, where $R_{n,m_1,m_2}(t)\in \mathbb{Z}[t]$ are defined in Lemma~\ref{lem:derivn}. The digit $f_i(q)$ is divisible by $q-1$ for $1 \le i \le \min\{p-2,r-1\}$. Additionally, we have
\begin{equation}\label{eq:supermu}
\sum_{d \mid n} \mu\left(\frac{n}{d}\right) a_{md}(1) \equiv 0 \bmod n^{1+\min\{p-2, r-1\}}.
\end{equation}
\end{thm}
\begin{remark}
When $p<r+1$, Theorem~\ref{thmexpansion} does not give us $r$ $[n]_q$-digits of $a_{nm}(q)$, but the proof shows that in any case the polynomial
\begin{equation*}
r(q) = \sum_{i=0}^{r-1} f_i(q) [n]_q^i \in \mathbb{Q}[q]
\end{equation*}
satisfies $r^{(j)}(\omega) = a^{(j)}_{nm}(\omega)$ for all $\omega \in \mu_n \setminus \{1\}$ and $0 \le j \le r-1$.
\end{remark}
Assume that the conditions of Theorem~\ref{thmexpansion} hold and consider the base-$[n]_q$ expansion
\begin{equation}\label{eq:amnsuperexp}
a_{nm}(q) \equiv \sum_{i=0}^{\min\{p-2,r-1\}} f_i(q) [n]_q^i \bmod [n]_q^{1+\min\{p-2,r-1\}}.
\end{equation}
We regard \eqref{eq:amnsuperexp} as a $q$-analogue of \eqref{eq:supermu}. Indeed, \eqref{eq:supermu} follows quickly by plugging $q=1$ in \eqref{eq:amnsuperexp}, see the proof of Theorem~\ref{thmexpansion} in \S\ref{proofexpansion}. 

To deduce Corollary~\ref{cor:superexamples} from Theorem~\ref{thm:superqexamples}, we only need the case $S=\mathbb{P}$, $r=3$ of Theorem~\ref{thmexpansion}, which is given by the following corollary. In order to compute $f_0$, $f_1$ and $f_2$ using the recursion \eqref{eq:digitsthmexp}, we need the following values of $R_{n,m_1,m_2}$: $R_{n,0,m_2}=\delta_{m_2, 0}$ (Kronecker delta), $R_{n,1,1}=n$, $R_{n,1,2}(t) = n(n-1)(t-1)-2nt$. 
\begin{cor}\label{corexpansion}
Assume that $\{ a_n(q) \}_{n \ge 1}$ satisfies the $q$-Gauss congruences of order $3$. Then for all $m \ge 1$, $n\ge 2$ with $(n,6)=1$ we have 
\begin{equation*}
a_{nm}(q) \equiv f_0(q) + f_1(q) [n]_q + f_2(q)[n]_q^2 \bmod [n]_q^3,
\end{equation*}
where $f_0,f_1,f_2 \in \mathbb{Z}[q]$ are the first three $[n]_q$-digits of $a_{nm}(q)$, given by
\begin{equation*}
\begin{split}
f_0(q) &= G_{g_0,n}(q), \qquad f_1(q) = \frac{q-1}{n} \left(G_{g_1,n}(q) - qG'_{g_0,n}(q) \right),\\
f_2(q) &= \frac{(q-1)^2}{2n^2} \left( G_{g_2,n}(q) +G''_{g_0,n}(q)q^2+(qG'_{g_0,n}(q)-G_{g_1,n}(q))(n-1)+2q(G'_{g_0,n}(q)-G'_{g_1,n}) \right),
\end{split}
\end{equation*}
and $g_0,g_1,g_2 \colon D_n \to \mathbb{C}$ are given by
\begin{equation*}
g_0(d)=a_{md}(1), \quad  g_1(d)=a'_{mn}(\omega_n^d)\omega_n^d, \quad g_2(d)=a''_{mn}(\omega_n^d)\omega_n^{2d}.
\end{equation*}
\end{cor}
\subsection{$q$-supercongruences}
Below we use that $G_{g,p^k}(1)=g(p^{k-1})$ when $p$ is a prime and $k \ge 1$, and in particular $G_{g,p}(q)=g(1)$.
\begin{thm}\label{thm:superqbinom}
Let $a_n(q) = {an \brack bn}_q$. For all $n\ge 1$ and $\omega \in \mu_n$, we have 
\begin{align}
\label{eq:an1b} a_n(\omega)&= a_{\frac{n}{\mathrm{ord}(\omega)}}(1),\\
\label{eq:an2b} \omega a'_n(\omega)&= \mathrm{ord}(\omega)^2 \binom{\frac{an}{\mathrm{ord}(\omega)}}{\frac{bn}{\mathrm{ord}(\omega)}} \frac{b(a-b)n^2}{2},\\
\label{eq:an3b}\omega^2 a''_n(\omega)&=
 \binom{\frac{an}{\mathrm{ord}(\omega)}}{\frac{bn}{\mathrm{ord}(\omega)}} b(a-b)n^2  \left(\frac{b(a-b)n^2 }{4} + \frac{an\cdot\mathrm{ord}(\omega)-5}{12}\right),
\end{align}
In particular, $a_n(q)$ satisfies the $q$-Gauss congruences of order $3$. Thus, in the notation of Corollary~\ref{corexpansion}, for all $m \ge 1$, $n\ge 2$ with $(n,6)=1$ we have
\begin{equation}\label{eq:qbinomexp}
a_{nm}(q) \equiv f_0(q) +  f_1(q) [n]_q + f_2(q) [n]_q^2 \bmod [n]_q^3.
\end{equation}
\end{thm}
Specializing \eqref{eq:qbinomexp} to $n=p^k$ ($p\ge 5$ a prime, $k \ge 1$), $m=1$ and $q=1$, we obtain \eqref{eq:conseqbinomq} (since $f_0(1)=a_{p^{k-1}}(1), f_1(1)=f_2(1)=0$). Theorem~\ref{thm:superqbinom} is the first result providing a $q$-analogue of 
\eqref{eq:conseqbinomq} for $k>1$. Previously, various $q$-analogues were found only for the case $k=1$. Clark \cite{clark1995} has shown that
\begin{equation*}
{an \brack bn}_q \equiv {a \brack b}_{q^{n^2}} \bmod \Phi^2_n(q),
\end{equation*}
where $\Phi_n(q)$ is the $n$-th cyclotomic polynomial, which coincides with $[n]_q$ for $n$ a prime. Andrews \cite[Thm.~3]{andrews1999} has shown that if $p$ is an odd prime, then
\begin{equation*}
{ap \brack bp}_q \equiv q^{(a-b)b\binom{p}{2}} {a \brack b}_{q^p} \bmod [p]^2_q.	
\end{equation*}
Straub \cite[Thm.~1]{straub2011}, building on a work of Shi and Pan \cite{shi2007}, proved that for any prime $p \ge 5$,
\begin{equation}\label{straubres}
{ap \brack bp}_q \equiv {a \brack b}_{q^{p^2}} - \binom{a}{b}b(a-b)\frac{p^2-1}{24}(q^p-1)^2 \bmod {[p]^3_q},
\end{equation}
which refines Clark's result for $n=p$. Cai and Garc\'{i}a-Pulgar\'{i}n \cite{cai2001} obtained some variants of \eqref{straubres} when $a=2$, $b=1$. When $m=1$ and $n=p \ge 5$ is a prime, \eqref{eq:qbinomexp} simplifies to
\begin{equation}\label{eq:qbinomk1}
{ap \brack bp}_q \equiv \binom{a}{b} + \frac{\binom{a}{b}b(a-b)p}{2}(q^p-1) +  \frac{\binom{a}{b}b(a-b)}{2}\left( \frac{b(a-b)}{4}p^2 + \frac{ap^2-5}{12}-\frac{p-1}{2} \right)(q^p-1)^2 \bmod [p]_q^3.
\end{equation}
Recently, Straub \cite[Thm.~2.2]{straub2019} extended \eqref{straubres} as follows:
\begin{equation}\label{straubres2}
{am \brack bm}_q \equiv {a \brack b}_{q^{m^2}} - \binom{a}{b}b(a-b)\frac{m^2-1}{24}(q^m-1)^2 \bmod {\Phi_m(q)^3}
\end{equation}
for all $m\ge 1$ with $(m,6)=1$. As $\Phi_{p^k}(q)=p$ when $p$ is a prime and $k \ge 1$, substituting $m=p^k$ in \eqref{straubres2} gives the congruence $\binom{ap^k}{bp^k} \equiv \binom{a}{b} \bmod p^3$. Although \eqref{eq:qbinomk1} does not imply Straub's results, in \S\ref{sec:altern} we explain how to derive \eqref{straubres2} quickly using our methods. Recently Zudilin computed ${am \brack bm}_q$ modulo a fourth power of $\Phi_m$ \cite[Thms.~1,2]{zudilin2019}, but we shall not pursue such higher congruences here.

The following theorem is proved in \S\ref{sec:superqapery}.
\begin{thm}\label{thm:aprey}
Let $a_n(q)$ be the sequence $g_n(q)$ defined in Theorem~\ref{thm:superqexamples}. For any $P \in \mathbb{C}[x,y]$, let
\begin{equation*}
a_{n,P} = \sum_{k=0}^{n}  \binom{n}{k}^2 \binom{n+k}{n}^2 P(n,k).
\end{equation*}
For all $n \ge 1$ and $\omega \in \mu_n$, we have
\begin{align}
\label{eq:an1} a_n(\omega)&= a_{\frac{n}{\mathrm{ord}(\omega)}}(1),\\
\label{eq:an2}\omega a'_n(\omega)&= \mathrm{ord}(\omega)^2 a_{\frac{n}{\mathrm{ord}(\omega)},2xy-y^2+f(x,y)},\\
\label{eq:an3}\omega^2 a''_n(\omega)&= \mathrm{ord}(\omega)^4 a_{\frac{n}{\mathrm{ord}(\omega)}, (2xy-y^2+f(x,y))^2+\frac{x^2y}{3}-\frac{(x-y)^2}{6}}+ \mathrm{ord}(\omega)^2 a_{\frac{n}{\mathrm{ord}(\omega)},\frac{x^2}{6}-2xy+y^2-f(x,y)}.
\end{align}
In particular, $a_n(q)$ satisfies the $q$-Gauss congruences of order $3$. Thus, in the notation of Corollary~\ref{corexpansion}, for all $m \ge 1$, $n\ge 2$ with $(n,6)=1$ we have
\begin{equation}\label{eq:qaperyexp}
a_{nm}(q) \equiv f_0(q) +  f_1(q) [n]_q + f_2(q) [n]_q^2 \bmod [n]_q^3.
\end{equation}
\end{thm}
Specializing \eqref{eq:qaperyexp} to $n=p^k$ ($p\ge 5$ a prime, $k \ge 1$) and $q=1$, we obtain \eqref{eq:conseqaperyq} (since $f_0(1)=a_{p^{k-1}m}(1), f_1(1)=f_2(1)=0$). The sequence $a_n(q)$ was studied by Krattenthaler, Rivoal and Zudilin \cite{krattenthaler2006} and Zheng \cite{zheng2011} in the case $f(x,y)=(x-y)^2$ and recently by Straub \cite{straub2019} for general $f$. A different $q$-analogue of the Ap{\'e}ry numbers was considered by Adamczewski, Bell, Delaygue and Jouhet \cite[Prop.~1.5]{adamczewski2017}. 

Theorem~\ref{thm:aprey} is the first result providing a $q$-analogue of 
\eqref{eq:conseqaperyq} for $k>1$. 
Straub \cite[Cor.~1.1]{straub2019} proved that for any $m \ge 1$ with $(m,6) =1$, we have
\begin{equation}\label{eq:aperyqelegant}
a_{nm}(q) \equiv a_{n}(q^{m^2}) - (q^m-1)^2 \frac{m^2-1}{12} n^2 a_{n}(1)  \bmod \Phi_m(q)^3.
\end{equation}
As $\Phi_{p^k}(q)=p$ when $p$ is a prime and $k \ge 1$, substituting $m=p^k$ in \eqref{eq:aperyqelegant} gives the congruence $a_{np^k}(1) \equiv a_n(1) \bmod p^3$.  
When $n=p \ge 5$ is a prime, \eqref{eq:qaperyexp} simplifies to
\begin{equation}\label{eq:aperyprimeq}
a_{mp}(q) \equiv a_m(1) + b_m (q^p-1)+c_m(q^p-1)^2  \bmod [p]_q^3
\end{equation}
for
\begin{equation*}
\begin{split}
b_m &= p a_{m,2xy-y^2+f(x,y)},\\
c_m &= \frac{1}{2}(p^2 a_{m,(2xy-y^2+f(x,y))^2+\frac{x^2y}{3}-\frac{(x-y)^2}{6} }+a_{m,\frac{x^2}{6}-2xy+y^2-f(x,y)}-(p-1)a_{m,2xy-y^2+f(x,y)}).
\end{split}
\end{equation*}
Although \eqref{eq:aperyprimeq} does not imply \eqref{eq:aperyqelegant}, in \S\ref{sec:altern} we explain how to derive \eqref{eq:aperyqelegant} using our methods.
\begin{remark}
Straub also allowed $f(x,y)$, in the definition of $a_n(q)$, to assume negative values, by working in the ring $\mathbb{Z}[q,q^{-1}]$ of Laurent polynomials.
\end{remark}

The following theorem is proved in \S\ref{sec:superqzeta}.
\begin{thm}\label{thm:zeta}
Let $a_n(q)$ be the sequence $h_n(q)$ defined in Theorem~\ref{thm:superqexamples}. For any $P \in \mathbb{C}[x,y,z]$, let
\begin{equation*}
a_{n,P} = \sum_{k,\ell} \binom{n}{k}^2 \binom{n}{\ell} \binom{k}{\ell} \binom{k+\ell}{n} P(n,k,\ell).
\end{equation*}
For all $n \ge 1$ and $\omega \in \mu_n$, we have 
\begin{align}
\label{eq:an1z} a_n(\omega)&= a_{\frac{n}{\mathrm{ord}(\omega)}}(1),\\
\label{eq:an2z} \omega a'_n(\omega)&= \mathrm{ord}(\omega)^2 a_{\frac{n}{\mathrm{ord}(\omega)},xz-y^2-z^2+\frac{3xy+yz-x^2}{2}+f(x,y,z)},\\
\label{eq:an3z}\omega^2 a''_n(\omega)&=
\mathrm{ord}(\omega)^4 a_{\frac{n}{\mathrm{ord}(\omega)}, Q_1}+ \mathrm{ord}(\omega)^2 a_{\frac{n}{\mathrm{ord}(\omega)},Q_2},
\end{align}
where
\begin{equation*}
\begin{split}
Q_1 & =  \left(xz-y^2-z^2+ \frac{3xy+yz-x^2}{2}+f(x,y,z)\right)^2+\frac{x^2y - xy^2 + 2xyz + y^2z - yz^2}{12}\\
& \qquad +\frac{-x^2+xy-y^2-z^2+zy+zx}{6}, \\
Q_2 &=
\frac{1}{12} (7x^2-17xy+12y^2-12xz-7zy+12z^2)-f(x,y,z).
\end{split}
\end{equation*}
In particular, $a_n(q)$ satisfies the $q$-Gauss congruences of order $3$. Thus, in the notation of Corollary~\ref{corexpansion}, for all $m \ge 1$, $n\ge 2$ with $(n,6)=1$ we have
\begin{equation}\label{eq:qzetaexp}
a_{nm}(q) \equiv f_0(q) +  f_1(q) [n]_q + f_2(q) [n]_q^2 \bmod [n]_q^3.
\end{equation}
\end{thm}
Specializing \eqref{eq:qzetaexp} to $n=p^k$ ($p \ge 5$ a prime, $k \ge 1$), we obtain the  supercongruence \eqref{eq:conseqzetaq} (since $f_0(1)=a_{p^{k-1}m}(1),f_1(1)=f_2(1)=0$). 
When $n=p\ge 5$ a prime, \eqref{eq:qzetaexp} simplifies to
\begin{equation*}
a_{mp}(q) \equiv a_m(1) + b_m (q^p-1)+c_m(q^p-1)^2  \bmod [p]_q^3
\end{equation*}
for
\begin{equation*}
\begin{split}
b_m &= p a_{m,xz-y^2-z^2+ \frac{3xy+yz-x^2}{2}+f(x,y,z)},\\
c_m &= \frac{1}{2}\left(p^2 a_{m,Q_1}+a_{m,Q_2}-(p-1)a_{m,xz-y^2-z^2+ \frac{3xy+yz-x^2}{2}+f(x,y,z)}\right).
\end{split}
\end{equation*}
In \S\ref{sec:altern} we show that for $m \ge 1$ with $(m,6)=1$, we also have the more elegant $q$-supercongruence
\begin{equation}\label{eq:zetaqelegant}
a_{nm}(q) \equiv a_{n}(q^{m^2})-(q^m-1)^2 \frac{m^2-1}{24} a_{n,x^2+xy-yz}  \bmod \Phi_m(q)^3.
\end{equation}
The proofs of Theorems~\ref{thm:superqbinom}, \ref{thm:aprey}, \ref{thm:zeta} involve differentiating the relevant sequences and evaluating them at roots of unity, by using the values of the derivatives of the $q$-binomial coefficients at roots of unity. These values are given in \S\ref{sec:derivativesofbinom}, and especially in Corollary~\ref{cor:binomder}, which might be of independent interest.
\section{Criteria for Gauss congruences}\label{seccritgauss}
Here we review some classical results on Gauss congruences, mostly for comparison with results we obtain on $q$-Gauss congruences. 
\begin{proposition}\label{gausscrit} \cite[Ch.~5,~Ex.~5.2(a)\text{ and its solution}]{stanley1999}
Let $\{a_n \}_{n \ge 1}$ be a sequence of integers. The following conditions are equivalent.
\begin{enumerate}
\item $\{a_n\}_{n \ge 1}$ satisfies the Gauss congruences
\item For all $n,k \ge 1$ and all primes $p$: $a_{p^k n} \equiv a_{p^{k-1}n} \bmod {p^k}$
\item $\exp(\sum_{n \ge 1} a_n x^n/n) \in \mathbb{Z}[[x]]$
\end{enumerate}
\end{proposition}
The next proposition generalizes the equivalence between the first two conditions in Proposition~\ref{gausscrit}.
\begin{proposition}\cite[Prop.~11]{almkvist2006}
Let $\{a_n\}_{n \ge 1}$ be a sequence of integers and let $m$ be a positive integer. The following conditions are equivalent.
\begin{enumerate}
\item For all $n \ge 1$: $\sum_{d \mid n} \mu(n/d) a_d \equiv 0 \bmod {n^m}$
\item For all $n,k \ge 1$ and primes $p$: $a_{p^k n} \equiv a_{p^{k-1}n} \bmod {p^{km}}$
\end{enumerate}
\end{proposition}
The following result is a corollary of the Lagrange inversion theorem.
\begin{proposition}\label{propnthcoeff}
Let $f \in \mathbb{Z}[[u]]$ with $f(0)=1$. The sequence $\{[u^n]f^n(u)\}_{n \ge 1}$ satisfies the Gauss congruences.
\end{proposition}
\begin{proof}
Let $a_n=[u^n]f^n(u)$. From \cite[Eq.~(3.8)]{gessel1980}, we have
\begin{equation}\label{lagrange}
\exp\big( \sum_{n \ge 1} a_n x^n / n \big) = \sum_{ n \ge 1} \left([u^{n-1}]f^n(u) /n\right) x^{n-1}.
\end{equation}
Since $(f^n)' = n\cdot f' \cdot f^{n-1}$, we have
\begin{equation}\label{diffiden}
[u^{n-1}]f^n(u) = \frac{1}{n-1}[u^{n-2}](f^n)'(u) = \frac{n}{n-1} [u^{n-2}]f'(u) f^{n-1}(u) .
\end{equation}
for all $n \ge 2$. From \eqref{diffiden}, it follows that $[u^{n-1}]\frac{f^n(u)}{n} \in \frac{1}{n}\mathbb{Z} \cap \frac{1}{n-1} \mathbb{Z} = \mathbb{Z}$, and so from  \eqref{lagrange} it follows that $\exp( \sum_{n \ge 1} a_n x^n / n) \in \mathbb{Z}[[x]]$. Proposition~\ref{gausscrit} applied to $\{a_n\}_{n \ge 1}$ concludes the proof of the proposition.
\end{proof}
\begin{cor}\label{corexamg}
The following sequences satisfy the Gauss congruences.
\begin{enumerate}
\item \cite{zarelua2008} $a_n = \mathrm{Tr}(A^n)$, where $A \in \mathrm{Mat}_{m}(\mathbb{Z})$.
\item $a_n = \binom{an}{n}$ and $a_n = \binom{an-1}{n}$, where $a \ge 2$ is an integer.
\item $a_n = \sum_{k=0}^{\lfloor \frac{n}{2} \rfloor} \binom{n}{k} \binom{n-k}{k}$.
\end{enumerate}
\end{cor}
\begin{proof}
For the first part, note that
\begin{equation*}
\exp\big(\sum_{n \ge 1} \frac{\mathrm{Tr}(A^n) x^n}{n}\big) = \exp\left( \mathrm{Tr}( -\ln(I-Ax))\right) = \frac{1}{\det (I-Ax)} = \sum_{i \ge 0} \mathrm{Tr} \mathrm{Sym}^i (A) x^i, 
\end{equation*} 
where $\mathrm{Sym}^i (A)$ is the $i$-th symmetric power of $A$. Thus, Proposition~\ref{gausscrit} implies that $\{\mathrm{Tr}(A^n)\}_{n \ge 1}$ satisfies the Gauss congruences.

For the second part, apply Proposition~\ref{propnthcoeff} with $f(x) = (1+x)^a$ and $f(x) = (1-x)^{-(a-1)}$. For the third part, apply Proposition~\ref{propnthcoeff} with $f(x) =1+x+x^2$.
\end{proof}
We also have a $p$-adic version of Proposition~\ref{gausscrit}.
\begin{proposition}
Let $\{a_n \}_{n \ge 1}$ be a sequence of integers. Fix a prime $p$, let $\mathbb{Z}_{p}$ be the ring of $p$-adic integers and $\mathbb{Q}_p$ be its fraction field. Set $F(x) = \exp(\sum_{n \ge 1} a_n x^n/n) \in \mathbb{Q}[[x]] \subseteq \mathbb{Q}_p[[x]]$. The following conditions are equivalent.
\begin{enumerate}
\item $\{a_n\}_{n \ge 1}$ satisfies the Gauss congruences with respect to $\{p\}$
\item $F(x) \in \mathbb{Z}_p[[x]]$
\item $F(x^p)/F(x)^p \in 1+px\mathbb{Z}_p[[x]]$
\end{enumerate}
\end{proposition}
\begin{proof}
The equivalence of the second and the third conditions is known as the Dieudonn\'{e}-Dwork criterion, see \cite[\S~VII.2.3]{robert2000}. The equivalence of the first and the second conditions follows from the proof of the equivalence of the second and third conditions in Proposition~\ref{gausscrit}. Indeed, following the proof of Proposition~\ref{gausscrit} but working in the ring $\mathbb{Z}_p$ instead of $\mathbb{Z}$, we see that $F(x) \in \mathbb{Z}_p[[x]]$ holds if and only if
\begin{equation}\label{eq:condqs}
a_{r^k n} \equiv a_{r^{k-1}n} \bmod r^k \mathbb{Z}_p.
\end{equation}
for all $n,k \ge 1$ and all primes $r$. Since a prime $r$ is invertible in $\mathbb{Z}_p$ whenever $r \neq p$, condition \eqref{eq:condqs} is non-trivial only for $r=p$, in which case it becomes
\begin{equation*}
a_{p^k n} \equiv a_{p^{k-1}n} \bmod {p^k}
\end{equation*}
for all $n,k \ge 1$, as needed.
\end{proof}
\subsection{Proof of Lemma~\ref{lem:equivs}}
The first part of the lemma is proved as follows.

$\Rightarrow$: Assume that $\{a_n\}_{n \ge 1}$ satisfies the Gauss congruences with respect to $S$. Then for any $p\in S$, we may choose $n=p^k$ in \eqref{weakneckcond} and obtain \eqref{primecond}, as needed.

$\Leftarrow$: Assume that \eqref{primecond} holds for all $p \in S$. Let $n \in \mathbb{N}_{S}$, and suppose that $n$ factors as $n=\prod_{i=1}^{r} p_i^{e_i}$. Let $m \ge 1$. Fix $i \in \{1,2,\ldots,r\}$. Set $S_n = \{ d: d \text{ divides }n, \mu(d) \neq 0\}$ and $T_n = \{ d\in S_n: p_i \nmid d\}$. We partition $S_n$ into a disjoint union of pairs: $S_n = \cup_{d \in T_n} \{ d,dp_i\}$. Then
\begin{equation}\label{eq:sumpairs}
\sum_{d \mid n} \mu(d) a_{nm/d} = \sum_{d \in S_n} \mu(d) a_{nm/d} =\sum_{d \in T_n} \mu(d) (a_{nm/d} - a_{nm/dp_i}).
\end{equation}
Each summand in the right-hand side of \eqref{eq:sumpairs} is divisible by $p_i^{e_i}$ by \eqref{primecond}, which shows that $\sum_{d \mid n} \mu(d) a_{nm/d}$ is divisible by $p_i^{e_i}$. Since $i$ was arbitrary,  $\sum_{d \mid n} \mu(d) a_{nm/d}$ is divisible by $n$, as needed.

We continue to the second part of the lemma.

$\Rightarrow$: Assume that $\{a_n\}_{n \ge 1}$ satisfies the Gauss congruences with respect to $\mathbb{P}$. Choosing $m=1$ in \eqref{weakneckcond}, we see that $\{a_n\}_{n \ge 1}$ satisfies the Gauss congruences.

$\Leftarrow$: Assume that $\{a_n\}_{n \ge 1}$ satisfies the Gauss congruences. By Proposition~\ref{gausscrit}, we have that $a_{p^kn} \equiv a_{p^{k-1}n} \bmod p^k$ for all $n,k \ge 1$ and all primes $p$. Let $m \ge 1$. Replacing $n$ with $nm$, we see that $a_{p^knm} \equiv a_{p^{k-1}nm} \bmod p^k$ for all $n, k \ge 1$ and all primes $p$. By another application of Proposition~\ref{gausscrit} it follows that the sequence $b_n:=a_{nm}$ satisfies the Gauss congruences, which gives us \eqref{weakneckcond} with fixed $m$ and for all $n \ge 1$. Since $m$ was arbitrary, it follows that \eqref{weakneckcond} holds with $S=\mathbb{P}$, as needed. \qed
\section{Criteria for $q$-Gauss congruences}\label{secchar}
\subsection{Auxiliary lemmas}
\begin{lem}\label{lem:mob}\cite[Ch.~2]{ireland2013}
\begin{enumerate}
	\item The divisor sum $\sum_{d \mid n} \mu(d)$ equals $1$ if $n=1$, and is $0$ otherwise.
	\item The M\"obius function is multiplicative, that is, $\mu(n_1 n_2) = \mu(n_1) \mu(n_2)$ whenever $(n_1,n_2)=1$.
\end{enumerate}
\end{lem}
\begin{lem}\label{lemrem}
Let $f(q) \in \mathbb{C}[q]$ and $n \ge 1$. Assume that as a function of $\omega \in \mu_n$, $f(\omega)$ depends only on the order of $\omega$. Then the remainder of $f(q)$ upon division by $[n]_q$ is 
\begin{equation*}
G_{g,n}(q)
\end{equation*}
for
\begin{equation*}
g\colon D_n \to \mathbb{C}, \quad g(d) = f(\omega_n^d).
\end{equation*}
\end{lem}
\begin{proof}
The degree of $G_{g,n}(q)$ is less than $\deg [n]_q$, since if $d \in D_n$ then $\deg \frac{[n]_q}{[n/d]_q} = n-\frac{n}{d} < n-1 = \deg [n]_q$. Let $\omega \in \mu_n \setminus \{1\}$. We have, for any $d$ dividing $n$, 
\begin{equation*}
\frac{[n]_q}{[n/d]_q} \Big|_{q = \omega} = \begin{cases} d & \text{if }d \mid \frac{n}{\mathrm{ord}(\omega)}, \\ 0 & \text{otherwise.}\end{cases}
\end{equation*}
Thus
\begin{equation}\label{eq:romega}
G_{g,n}(\omega) = \sum_{d \mid \frac{n}{\mathrm{ord}(\omega)}} \sum_{e \mid d} \mu(\frac{d}{e})f(\omega_n^e).
\end{equation}
Changing the order of summation in \eqref{eq:romega}, we obtain
\begin{equation*}
G_{g,n}(\omega) = \sum_{e \mid \frac{n}{\mathrm{ord}(\omega)}} f(\omega_n^e) \sum_{d: e \mid d \mid \frac{n}{\mathrm{ord}(\omega)}}  \mu(\frac{d}{e})= \sum_{e \mid \frac{n}{\mathrm{ord}(\omega)}} f(\omega_n^e) \sum_{d': d' \mid \frac{n}{\mathrm{ord}(\omega)e}}  \mu(d'),
\end{equation*}
which equals $f(\omega_n^{\frac{n}{\mathrm{ord}(\omega)}})=f(\omega)$ by the first part of Lemma~\ref{lem:mob}. This implies that $f(q)-G_{q,n}(q)$ is divisible by $[n]_q$, as needed.
\end{proof}

\subsection{Proof of Proposition~\ref{propcharweak}}
The equivalence of the second and the third conditions in Proposition~\ref{propcharweak} follows from a general observation: a polynomial $F(q) \in \mathbb{C}[q]$ is divisible by $\Phi_{k}(q)$ if and only if $F(\omega)=0$ for any primitive root of unity $\omega$ of order $k$. We turn to prove the equivalence of the first and the second conditions.

$\Leftarrow$:  Assume that $\{a_n(q)\}_{n \ge 1} \subseteq \mathbb{Z}[q]$ satisfies \eqref{charweak}, that is,
\begin{equation}\label{charweak3}
a_{nm}(\omega_n^i) =a_{m(n,i)}(1).
\end{equation}
for all $n \in \mathbb{N}_S$ and $m,i\ge 1$. We establish \eqref{qweakneckcond}, which may be stated as follows:
\begin{equation}\label{qweakneckcond2}
\sum_{d \mid n} \mu(d) a_{nm/d}(\omega_n^{id})=0.
\end{equation}
for all $n \in \mathbb{N}_{S}$ and $m,i \ge 1$ with $n \nmid i$. We simplify \eqref{qweakneckcond2} using \eqref{charweak3} as follows:
\begin{equation}\label{eq:sumsinduct}
\sum_{d \mid n} \mu(d) a_{nm/d}(\omega_n^{id}) = \sum_{d \mid n} \mu(d) a_{\frac{n}{d} m}(\omega_{n/d}^i)= \sum_{d \mid n} \mu(d) a_{(\frac{n}{d},i) m}(1).
\end{equation}
If $a_f(1)$ appears in the right-hand side of \eqref{eq:sumsinduct}, then $f=f'm$ for some $f' \mid (n,i)$. For any $f' \mid (n,i)$, the term $a_{f'm}(1)$ appears in the right-hand side of \eqref{eq:sumsinduct} with coefficient
\begin{equation}\label{eq:coeffmob}
\sum_{\substack{d \mid n\\ (\frac{n}{d},i)=f'}} \mu(d) = \sum_{\substack{d \mid \frac{n}{f'}\\ (\frac{n}{df'},\frac{i}{f'})=1}} \mu(d).
\end{equation}
Let $g$ be the largest divisor of $\frac{n}{f'}$ which is divisible only by primes dividing $\frac{i}{f'}$. The condition $(\frac{n}{df'},\frac{i}{f'})=1$ is equivalent to $g \mid d$. Also let $\tilde{(\frac{n}{gf'})}$ denote the largest factor of $\frac{n}{gf'}$ coprime to $g$. Using Lemma~\ref{lem:mob}, the sum in the right-hand side of \eqref{eq:coeffmob} is
\begin{equation}\label{eq:mobuse}
\sum_{d' \mid \frac{n}{gf'}} \mu(gd') = \mu(g) \sum_{d' \mid \tilde{(\frac{n}{gf'})}}\mu(d') = \mu(g) 1_{\tilde{(\frac{n}{gf'})} = 1}.
\end{equation}
We now explain why the sum in \eqref{eq:coeffmob} is necessarily $0$. Otherwise, by \eqref{eq:mobuse}, $g$ must be squarefree and every prime factor of $\frac{n}{gf'}$ must be a factor of $g$. In particular, every prime factor of $\frac{n}{f'}$ divides $g$. Combined with the fact $g$ is squarefree and the definition of $g$, it follows that $g=\frac{n}{f'}$. Again, by the definition of $g$, every prime factor of the squarefree number $g=\frac{n}{f'}$ divides $\frac{i}{f'}$, and thus $n$ divides $i$, a contradiction. Thus, the sum in \eqref{eq:sumsinduct} is also $0$, as needed.

$\Rightarrow$:  Assume that $\{a_n(q)\}_{n \ge 1} \subseteq \mathbb{Z}[q]$ satisfies the $q$-Gauss congruences with respect to $S$. We show by induction on $n \in \mathbb{N}_S$ that \eqref{qweakneckcond2} implies \eqref{charweak3}. For $n=1$, \eqref{charweak3} is a tautology. We assume that \eqref{charweak3} holds for all $n\in \mathbb{N}_S$ smaller than $k\in \mathbb{N}_S$, and prove it for $n=k$. If $i$ is divisible by $k$, there is nothing to prove. Otherwise, if $k$ does not divide $i$, we have from \eqref{qweakneckcond2} that
\begin{equation}\label{qneckwithroot}
\sum_{d \mid k}\mu(d) a_{km/d}(\omega^{id}_k) =0
\end{equation}
whenever $k \nmid i$. The induction hypothesis tells us that for any $d \neq 1$ dividing $k$, 
\begin{equation}\label{inductell}
a_{km/d}(\omega^{id}_{k}) =a_{km/d}(\omega^i_{k/d})= a_{m(k/d,i)}(1).
\end{equation}
From \eqref{qneckwithroot} and \eqref{inductell} we obtain
\begin{equation}\label{qneckwithroot2}
a_{km}(\omega_k^i) + \sum_{d \mid k,\, d \neq 1}\mu(d) a_{(k/d,i)}(1) =0.
\end{equation}
We need to prove that $a_{km}(\omega_k^i)=a_{m(k,i)}(1)$, which, using \eqref{qneckwithroot2}, becomes the following equivalent condition:
\begin{equation*}
\sum_{d \mid k}\mu(d) a_{(k/d,i)m}(1) =0,
\end{equation*}
which was established in the other direction of the proof by showing that the coefficient of $a_{f'm}(1)$ (where $f'  \mid (k,i)$) is $0$, so we are done. \qed
\subsection{Proof of Lemma~\ref{lem:equivsq}}
The first part of the lemma follows immediately by choosing $n=p^k$ for $p \in S$ in \eqref{qweakneckcond}. We turn to the proof of the second part of the lemma.

$\Rightarrow$: Suppose that $\{ a_n(q)\}_{n \ge 1}$ satisfies the $q$-Gauss congruences with respect to $\mathbb{P}$. Then by choosing $m=1$ in \eqref{qweakneckcond} we see that $\{ a_n(q)\}_{n \ge 1}$ satisfies the $q$-Gauss congruences, as needed. 

$\Leftarrow$: Suppose that $\{ a_n(q) \}_{n \ge 1}$ satisfies the $q$-Gauss congruences. By Proposition~\ref{propcharweak}, it suffices to prove that 
\begin{equation*}
a_n(\omega_n^i)=a_{(n,i)}(1).
\end{equation*}
for all $n,i \ge 1$. In other words, we need to deduce \eqref{qweakneckcond2} from \eqref{charweak3}, but with $m$ fixed and equal to $1$ (and $S=\mathbb{P}$). In Proposition~\ref{propcharweak}, it is established that \eqref{charweak3} implies \eqref{qweakneckcond2}, and following the proof we see that in fact $m$ can be fixed during it, so we are done. \qed
\subsection{Proof of Proposition~\ref{proprem}}
According to Proposition~\ref{propcharweak}, we may apply Lemma~\ref{lemrem} with $f(q)=a_{nm}(q)$, which establishes the proposition since $f(\omega_{n}^e)=a_{nm}(\omega_{n}^{e}) = a_{nm/(n/e)}(1)=a_{em}(1)$ if $e \mid n$. \qed
\section{Examples}
To verify our examples we need two results. The first is a standard result  \cite[Ch.~3,~Ex.~45(b)]{stanley1997} (cf. \cite{slavin2008}).
\begin{lem}\label{lemroots}
Let $n,k,d$ be non-negative integers. We have
\begin{equation*}
{n \brack k}_{\omega_n^d} = \begin{cases} \binom{(n,d)}{(n,d)k/n} & \mbox{if $n \mid dk$,} \\ 0 & \mbox{otherwise.} \end{cases}
\end{equation*}
\end{lem}
\begin{proof}
Plugging $q = \omega_n^d$ in \eqref{qbinom}, we obtain
\begin{equation}\label{binomqsub}
(1-(-t)^{n/(n,d)})^{(n,d)} = \sum_{k=0}^{n} {n \brack k}_{\omega_n^d} t^k \omega_n^{d\binom{k}{2}}.
\end{equation}
Comparing the coefficients of $t^k$ on both sides of \eqref{binomqsub}, we conclude the proof of the lemma.
\end{proof}
We also need the following lemma.
\begin{lem}\label{lemmatricessim}
Let $n \ge 1$ and let $\omega$ be a primitive root of unity of order $n$. Let 
\begin{equation*}
A_{\omega}(t) = \begin{bmatrix} 1  & \omega^{n-1} \\ t & 0 \end{bmatrix} \begin{bmatrix} 1  & \omega^{n-2} \\ t & 0 \end{bmatrix} \cdots \begin{bmatrix} 1  & 1 \\ t & 0 \end{bmatrix} \in \mathrm{Mat}_2(\mathbb{Z}[\omega][t])
\end{equation*}
and 
\begin{equation*}
A(t) = \begin{bmatrix} 1  & t \\ 1 & 0 \end{bmatrix} \in \mathrm{Mat}_2(\mathbb{Z}[t]).
\end{equation*}
Then $A_{\omega}(t)$, $A(t^{n})$ have the same characteristic polynomial.
\end{lem}
\begin{proof}
The characteristic polynomial of $A(t^n)$ is $X^2 - X-t^n$, so it suffices to show that
\begin{equation*}
\det(A_{\omega}(t)) = -t^n, \quad \mathrm{Tr}(A_{\omega}(t)) = 1.
\end{equation*}
By multiplicativity of the determinant, we have 
\begin{equation*}
\det(A_{\omega}(t)) = \prod_{i=0}^{n-1} (-\omega^i t) = t^n (-1)^n \omega^{\binom{n}{2}} = -t^n.
\end{equation*}
Let $P(t) = \mathrm{Tr}(A_{\omega}(t))$. We have
\begin{equation}\label{pat0}
P(0) = \mathrm{Tr}\left( \begin{bmatrix} 1  & \omega^{n-1} \\ 0 & 0 \end{bmatrix} \begin{bmatrix} 1  & \omega^{n-2} \\ 0 & 0 \end{bmatrix}  \cdots \begin{bmatrix} 1  & 1 \\ 0 & 0 \end{bmatrix}     \right) = \mathrm{Tr}\left( \begin{bmatrix} 1  & 1 \\ 0 & 0 \end{bmatrix}   \right)= 1. 
\end{equation}
Let $\omega_2 \in \mu_n$. By conjugating $A_{\omega}(\omega_2)$ with $\mathrm{Diag}(1,\omega_2)$ and using the property $\mathrm{Tr}(XY)=\mathrm{Tr}(YX)$, we see that
\begin{equation}\label{patom}
\begin{split}
P(\omega_2) &= \mathrm{Tr} \left(\begin{bmatrix} 1  & \omega^{n-1}\omega_2 \\ 1 & 0 \end{bmatrix} \begin{bmatrix} 1  & \omega^{n-2}\omega_2 \\ 1 & 0 \end{bmatrix} \cdots \begin{bmatrix} 1  & \omega_2 \\ 1 & 0 \end{bmatrix}\right) \\
&=  \mathrm{Tr} \left(\begin{bmatrix} 1  & \omega^{n-1} \\ 1 & 0 \end{bmatrix} \begin{bmatrix} 1  & \omega^{n-2} \\ 1 & 0 \end{bmatrix} \cdots \begin{bmatrix} 1  & 1 \\ 1 & 0 \end{bmatrix}\right) = P(1).
\end{split}
\end{equation}
Plugging $q = \omega$ in \eqref{eq:panen}, we see that $P(1)=1$. From \eqref{pat0} and \eqref{patom}, the polynomial $P$ is of degree $\le n$ and assumes the value $1$ $n+1$ times. Thus, $P$ is the constant polynomial $1$, as needed.
\end{proof}
\subsection{Simple examples}\label{subsecweakthm}
Here we verify that the examples given in \S\ref{secrel} satisfy the $q$-Gauss congruences. We start with $a_n(q) = {an \brack bn}_q$. By Corollary~\ref{cor:sp}, it suffices to show that
\begin{equation}\label{congprimesstep}
{an \brack  bn}_{\omega_n^i} = \binom{a(n,i)}{b(n,i)}.
\end{equation}
By Lemma~\ref{lemroots}, the left-hand side of \eqref{congprimesstep} is equal to ${an \brack  bn}_{\omega_{an}^{ai}} = \binom{(an,ai)}{(an,ai)b/a} = \binom{a(n,i)}{b(n,i)}$, as needed. We now consider $b_n(q) = {an-1 \brack bn}_q$, for which we have to show that
\begin{equation*}
{an-1 \brack  bn}_{\omega_n^i} = \binom{a(n,i)-1}{ b(n,i)}.
\end{equation*}
This equality can be deduced from \eqref{congprimesstep} since
\begin{equation*}
{am-1 \brack bm}_{q}  = {am \brack bm}_q  \frac{[(a-b)m]_q}{[am]_q}
\end{equation*}
and if $\omega^m =1$ then
\begin{equation*}
\lim_{q \to \omega }\frac{[(a-b)m]_q}{[am]_q} = \frac{a-b}{a}.
\end{equation*}
We continue with $c_n(q) = [t^{bn}]\prod_{i=0}^{n-1}(1-tq^i)^{a}$. By Corollary~\ref{cor:sp}, we need to prove that $c_n(\omega_n^k) = c_{(n,k)}(1)$, that is,
\begin{equation}\label{compcoeffs}
[t^{bn}]\prod_{i=0}^{n-1} (1-t \omega_n^{ki})^a = [t^{b(n,k)}](1-t)^{a(n,k)}.
\end{equation}
The left-hand side of \eqref{compcoeffs} may be evaluated as follows:
\begin{equation*}
\begin{split}
[t^{bn}]\prod_{i=0}^{n-1} (1-t \omega_n^{ki})^a &= [t^{bn}] \Big(\prod_{i=0}^{\frac{n}{(n,k)}-1} (1-t \omega_{n/(n,k)}^{ik/(n,k)}) \Big)^{a(n,k)}\\
&=  [t^{bn}] ( 1-t^{n/(n,k)})^{a(n,k)} \\
&=  [t^{b(n,k)}] (1-t)^{a(n,k)},
\end{split}
\end{equation*}
as needed.
\subsection{Proof of Theorem~\ref{propweakbinom}}
We prove both parts using Corollary~\ref{cor:sp}. We start with $d_n(q) = \sum_{i=0}^{\lfloor \frac{n}{2} \rfloor} q^{i(i+b)}{n \brack i}_q {n-i \brack i}_q$. We need to prove that $d_n(\omega_n^k) = d_{(n,k)}(1)$. By Lemma~\ref{lemroots},
\begin{equation*}
\begin{split}
d_n(\omega_n^k)&=\sum_{\substack{0 \le i \le \lfloor \frac{n}{2} \rfloor \\ n \mid ik}} \omega_n^{ki(i+b)} \binom{(n,k)}{(n,k)i/n} {n-i \brack i}_{\omega_n^k}\\
&=\sum_{\substack{0 \le i \le \lfloor \frac{n}{2} \rfloor \\ \frac{n}{(n,k)} \mid i}} \binom{(n,k)}{(n,k)i/n} {n-i \brack i}_{\omega_{n-i}^{k(n-i)/n}}\\
&=\sum_{\substack{0 \le i \le \lfloor \frac{n}{2} \rfloor \\ \frac{n}{(n,k)} \mid i}} \binom{(n,k)}{(n,k)i/n} \binom{(n-i,k(n-i)/n)}{(n-i,k(n-i)/n)i/(n-i)}.
\end{split}
\end{equation*}
Since $(n-i,k(n-i)/n)=(\frac{n-i}{n/(n,k)}\frac{n}{(n,k)},\frac{n-i}{n/(n,k)}\frac{k}{(n,k)})=\frac{n-i}{n/(n,k)}$, we may simplify the last sum as
\begin{equation*}
\begin{split}
d_n(\omega_n^k)&=\sum_{\substack{0 \le i \le \lfloor \frac{n}{2} \rfloor \\ \frac{n}{(n,k)} \mid i}} \binom{(n,k)}{(n,k)i/n} \binom{(n,k)(n-i)/n}{(n,k)i/n}\\
&=\sum_{\substack{0 \le i' \le \lfloor \frac{n}{2} \rfloor/(n/(n,k)) }} \binom{(n,k)}{i'} \binom{(n,k)-i'}{i'} =d_{(n,k)}(1).
\end{split}
\end{equation*}
We now prove the theorem for $e_n(q)$. Let 
\begin{equation*}
B_n(q) = A(q^{n-1}) A(q^{n-2}) \cdots A(1).
\end{equation*}
Since $i \mapsto (\omega_n^{k})^i=\omega_{n/(n,k)}^{ik/(n,k)}$ has period $n/(n,k)$, we have
\begin{equation}\label{bnpoweriden}
B_n(\omega_n^k) = B_{n/(n,k)}^{(n,k)}(\omega_{n/(n,k)}^{k/(n,k)}).
\end{equation}
If $(a,b)=1$, then Lemma~\ref{lemmatricessim} with $t=1$, $\omega = \omega_a^b$ and $n=a$ implies that $B_a(\omega_a^b)$ and $B_1(1)$ have the same characteristic polynomial, and so
\begin{equation}\label{b1trace}
\mathrm{Tr}(B^j_a(\omega_a^b)) = \mathrm{Tr}(B^j_1(1))
\end{equation}
holds for all $j$ and $a,b$ with $(a,b)=1$. By Corollary~\ref{cor:sp}, we need to prove that $e_n(\omega_n^k) = e_{(n,k)}(1)$, that is,
\begin{equation}\label{twotrace}
\mathrm{Tr}(B_n(\omega_n^k)) = \mathrm{Tr}(B^{(n,k)}_1(1)).
\end{equation}
From \eqref{b1trace} with $a=n/(n,k)$, $b=k/(n,k)$ and $j=(n,k)$, we obtain that the right-hand side of \eqref{twotrace} is $\mathrm{Tr}(B_{n/(n,k)}^{(n,k)}(\omega_{n/(n,k)}^{k/(n,k)}))$, which in turn equals the left-hand side of \eqref{twotrace} according to  \eqref{bnpoweriden}. \qed
\section{Proof of Theorem~\ref{thm:propcsp}}\label{seccsp}
We begin with the triple $(B_{n,k},\mathbb{Z}/n\mathbb{Z},{n;b;q \brack k}_2)$. The polynomial ${m_1 \brack m_2}_q$ has non-negative coefficients for all $m_1 \ge m_2 \ge 0$ (as follows from \eqref{qbinom}, for instance), and so 
${n;b;q \brack k} =\sum_{i =0}^{\lfloor \frac{n-k}{2} \rfloor} q^{i(i+b)} {n \brack i}_q {n-i \brack i+k}_q$ must also have non-negative coefficients. For any $0 \le i \le n$, the product $\binom{n}{i}\binom{n-i}{i+k}$ is the number of words in $B_{n,k}$ with $i$ $2$-s and $i+k$ $0$-s, and so
\begin{equation}\label{combtrin}
{n;b;1 \brack k} = \sum_{i \ge 0} \binom{n}{i} \binom{n-i}{i+k} = |B_{n,k}|.
\end{equation}
Given $g \in \mathbb{Z}/n\mathbb{Z}$, the set $B_{n,k}^g$ consists of elements of $B_{n,k}$ with period $(g,n)$, that is, of words of the form
\begin{equation*}
w^{\frac{n}{(g,n)}}= \underbrace{w \mid w \mid w \mid \cdots \mid w}_{n/(g,n)},
\end{equation*}
where $|$ denotes concatenation, and the length of $w$ is $(g,n)$. For $w^{\frac{n}{(g,n)}}$ to be in $B_{n,k}$, it is necessary and sufficient that $w \in B_{(g,n), k(g,n)/n}$ (in particular, $kg \equiv 0 \bmod n$). Thus,
\begin{equation}\label{sizeoffixtrin}
|B_{n,k}^g| = \begin{cases} B_{(g,n), k(g,n)/n} & \text{if }kg \equiv 0 \bmod n,\\ 0 & \text{otherwise.}\end{cases}
\end{equation}
To verify that $(B_{n,k},\mathbb{Z}/n\mathbb{Z},{n;b;q \brack k})$ exhibits the CSP, we need to prove that for all $g,g' \in \mathbb{Z}/n\mathbb{Z}$ with $\gcd(n,g)=\gcd(n,g')$, 
\begin{equation}\label{csptrin}
{n;b;\omega_n^{g'} \brack k} = |B_{n,k}^{g}|.
\end{equation}
By \eqref{combtrin} and \eqref{sizeoffixtrin}, the right-hand side of \eqref{csptrin} is ${(g,n);b;1 \brack k(g,n)/n}$ if $kg \equiv 0 \bmod n$, and $0$ otherwise. Thus, \eqref{csptrin} is equivalent to
\begin{equation*}
{n;b;\omega_n^{g'} \brack k} = \begin{cases} {(g,n);b;1 \brack k(g,n)/n} & \text{if }kg \equiv 0 \bmod n, \\ 0 & \text{otherwise.} \end{cases}
\end{equation*}
To prove this, we use Lemma~\ref{lemroots}, which implies that
\begin{equation}\label{compcoeffs2k}
{n;b;\omega_n^{g'} \brack k}=\sum_{\substack{0 \le i \le \lfloor \frac{n-k}{2} \rfloor \\ n/(g',n) \mid i}}  \binom{(g',n)}{i/(n/(g',n))} {n-i \brack i+k}_{\omega_{n-i}^{g'(n-i)/n}}.
\end{equation}
If $kg \neq 0 \bmod n$, then $kg' \neq 0 \bmod n$ also and Lemma~\ref{lemroots} implies that ${n-i \brack i+k}_{\omega_{n-i}^{g'(n-i)/n}} = 0$ whenever $n/(g',n) \mid i$ and so ${n;b;\omega_n^{g'} \brack k} =0$, as needed. Otherwise, Lemma~\ref{lemroots} tells us that ${n-i \brack i+k}_{\omega_{n-i}^{g'(n-i)/n}} = \binom{(g',n) - i/(n/(g',n))}{i/(n/(g',n)) + k(g',n)/n}$, and so the sum in \eqref{compcoeffs2k} is exactly ${(g',n);b;1 \brack k(g',n)/n}={(g,n);b;1 \brack k(g,n)/n}$, as needed.

We turn to the triple $(C_{n,k},\mathbb{Z}/n\mathbb{Z},e_{n,k}(q))$. The entries of $A(q^i,t)$ are polynomials in $q$ and $t$ with non-negative coefficients, and so $e_{n,k}(q)$ must also have non-negative coefficients. Set
\begin{equation*}
e_n(q,t) = \mathrm{Tr}(A(q^{n-1},t)A(q^{n-2},t) \cdots A(1,t)).
\end{equation*}
By definition, $e_{n,k}(q) = [t^k]e_n(q,t)$. Let $S_{n+1,k}$ be the set of words $w$ of length $n+1$ on letters $0,1$, with no consecutive $1$-s, and with $k$ indices $1 \le i \le n$ such that $w_{i}=1,w_{i+1}=0$. For all $w \in \cup_{k=0}^{n} S_{n+1,k}$, set
\begin{equation*}
W_3(w) = \sum_{\substack{1 \le a \le n: \\ w_a=0, w_{a+1}=1}} (n-a).
\end{equation*}
A direct inductive argument shows that for all $n \ge 1$ and $i,j \in \{0,1\}$, we have
\begin{equation}\label{eq:ijval}
(A(q^{n-1},t)A(q^{n-2},t) \cdots A(1,t))_{i,j} = \sum_{k=0}^{n} t^k \sum_{\substack{w \in S_{n+1,k} \\ w_1=i,w_{n+1}=j}} q^{W_3(w)}.
\end{equation}
Let $S^{'}_{n+1,k}$ be the subset of $S_{n+1,k}$ consisting of words that start and end with the same letter. Then \eqref{eq:ijval} implies that
\begin{equation}\label{eq:enqtviasnk}
e_n(q,t) = \sum_{k=0}^{n} t^k \sum_{w \in S^{'}_{n+1,k}}q^{W_3(w)}.
\end{equation}
By removing the first letter of each word in $S^{'}_{n+1,k}$, we obtain a set of the same size, namely $C_{n,k}$. Thus, \eqref{eq:enqtviasnk} implies that
\begin{equation}\label{enkviacnk}
e_{n,k}(q) =  \sum_{w \in S^{'}_{n+1,k}}q^{W_3(w)} = \sum_{w \in C_{n,k}} q^{W_1(w)}.
\end{equation}
In particular,
\begin{equation}\label{gncase}
e_{n,k}(1) = |C_{n,k}|.
\end{equation}
Given $g \in \mathbb{Z}/n\mathbb{Z}$, the set $C_{n,k}^g$ consists of elements of $C_{n,k}$ with period $(g,n)$, that is, of words of the form
\begin{equation*}
w^{\frac{n}{(g,n)}}= \underbrace{w \mid w \mid w \mid  \cdots \mid w}_{n/(g,n)},
\end{equation*}
where $|$ denotes concatenation, and the length of $w$ is $(g,n)$. For $w^{\frac{n}{(g,n)}}$ to be in $C_{n,k}$, it is necessary and sufficient that $kg \equiv 0 \bmod n$ and $w \in C_{(g,n),k(g,n)/n}$. Thus,
\begin{equation}\label{sizeoffixlucas}
|C^g_{n,k}| = \begin{cases} |C_{(g,n),k(g,n)/n}| & \text{if }kg \equiv 0 \bmod n, \\ 0 & \text{otherwise.} \end{cases}
\end{equation}
To verify that $(C_{n,k},\mathbb{Z}/n\mathbb{Z},e_{n,k}(q))$ exhibits the CSP, we need to prove that for all $g,g' \in \mathbb{Z}/n\mathbb{Z}$ with $\gcd(n,g)=\gcd(n,g')$,
\begin{equation}\label{cspluc} e_{n,k}(\omega^{g'}_n) = |C_{n,k}^g|.
\end{equation}
If we set $\omega = \omega_n^{g'}=\omega_{n/(g',n)}^{g'/(g',n)}$, then
\begin{equation}\label{eq:prodasatw}
A(q^{n-1},t)A(q^{n-2},t)\cdots A(1,t)\Big|_{q=\omega} = \left(A(\omega^{\frac{n}{(g',n)}-1},t)A(\omega^{\frac{n}{(g',n)}-2},t) \cdots A(1,t)\right)^{(g',n)}.
\end{equation}
Setting
\begin{equation*}
A_{\omega}(t) = A(\omega^{\frac{n}{(g',n)}-1},t)A(\omega^{\frac{n}{(g',n)}-2},t) \cdots A(1,t),
\end{equation*}
we obtain from \eqref{eq:prodasatw} that
\begin{equation}\label{enkdif}
e_{n,k}(\omega_n^{g'}) = [t^k] \mathrm{Tr}(A^{(g',n)}_{\omega}(t))= [t^k] \mathrm{Tr}(A^{(g,n)}_{\omega}(t)).
\end{equation}
By Lemma~\ref{lemmatricessim}, $A_{\omega}(t)$ and $A(t^{\frac{n}{(g,n)}},1)$ have the same characteristic polynomial. Thus, \eqref{enkdif} implies that
\begin{equation}\label{enkdif2}
e_{n,k}(\omega_n^{g'}) = [t^k] \mathrm{Tr}(A(t^{\frac{n}{(g,n)}},1)^{(g,n)}).
\end{equation}
If $kg \neq 0 \bmod n$, then \eqref{enkdif2} and \eqref{sizeoffixlucas} show that \eqref{cspluc} holds in this case. If $kg \equiv 0 \bmod n$, then \eqref{gncase}, \eqref{sizeoffixlucas} and \eqref{enkdif2} imply that
\begin{equation*}
e_{n,k}(\omega_n^{g'}) = [s^{\frac{k(g,n)}{n}}] \mathrm{Tr}(A(s,1)^{(g,n)}) =e_{(g,n),k(g,n)/n}(1) = |C^g_{n,k}|,
\end{equation*}
that is, \eqref{cspluc} again holds, as needed. \qed

\section{Criteria for supercongruences and $q$-Gauss congruences of order $d$}
\subsection{Auxiliary results}
We define the degree of the zero polynomial to be $-\infty$.
\begin{lem}\label{lem:derivn}
Let $n$ be a positive integer and let $\omega \in \mu_n \setminus \{ 1\}$.
\begin{enumerate}
\item Let $i \in \mathbb{Z}_{\ge 0}$. We have 
\begin{equation}\label{eq:qnderiv}
[n]^{(i)}_{\omega} = \frac{P_{n,i}(\omega)}{(\omega-1)^i\omega^i}
\end{equation}
for
\begin{equation*}
P_{n,i}(t) = i! \sum_{0,i-n \le j \le i-1} \binom{n}{i-j} (-t)^j (t-1)^{i-j-1} \in \mathbb{Z}[t].
\end{equation*}
\item Let $i,j \in \mathbb{Z}_{\ge 0}$. We have
\begin{equation*}
([n]_{\omega}^i)^{(j)} = \frac{R_{n,i,j}(\omega)}{(\omega-1)^j\omega^j}
\end{equation*}
for 
\begin{equation*}
R_{n,i,j}(t) = \sum_{a_1+\ldots+a_i = j} \binom{j}{a_1,\ldots,a_i} \prod_{1 \le k \le i} P_{n,a_k}(t) \in \mathbb{Z}[t].
\end{equation*}
Moreover, $\deg R_{n,i,j} \le j-i$ if $j \ge i$ and $R_{n,i,j}=0$ otherwise. Also, $R_{n,i,i} = i!n^i$.
\end{enumerate}
\end{lem}
\begin{proof}
To prove the first part of the lemma, recall the general Leibniz rule
\begin{equation*}
(f_{1}f_{2}\cdots f_{m_1})^{(m_2)}=\sum _{k_{1}+k_{2}+\ldots +k_{m_1}=m_2}\binom{m_2}{k_{1},k_{2},\ldots ,k_{m_1}}\prod_{1\leq j\leq m_1}f_{j}^{(k_{j})}.
\end{equation*}
Applying this rule with $f_1 = q^n-1$, $f_2 = \frac{1}{q-1}$, $m_1 = 2$ and $m_2 = i$, we obtain the following identity of rational functions:
\begin{equation}\label{eq:leib}
\begin{split}
[n]_q^{(i)} &= ( (q^n-1) \frac{1}{q-1})^{(i)}\\
&= \sum_{j=0}^{i} \binom{i}{j}(q^n-1)^{(i-j)} (\frac{1}{q-1})^{(j)}\\
&= \sum_{j=0}^{i-1} \binom{i}{j} n(n-1)\cdots (n-(i-j-1)) q^{n-(i-j)} \frac{j!(-1)^j}{(q-1)^{j+1}} +\frac{(q^n-1)i!(-1)^i}{(q-1)^{i+1}}.
\end{split}
\end{equation}
Plugging $q=\omega$ in \eqref{eq:leib}, we obtain
\begin{equation*}
[n]_{\omega}^{(i)} =  \frac{i!}{(\omega-1)^i \omega^i} \sum_{0,i-n \le j \le i-1} \binom{n}{i-j} (-\omega)^j (\omega-1)^{i-j-1},
\end{equation*}
as needed. To prove the second part of the lemma, we again apply the general Leibniz rule and obtain
\begin{equation}\label{eq:leib2}
([n]^i_q)^{(j)} = \sum_{a_1+\ldots+a_i = j} \binom{j}{a_1,\ldots,a_i} \prod_{1 \le k \le i} [n]_q^{(a_k)}.
\end{equation}
Using the first part of the lemma, \eqref{eq:leib2} may be written as follows when we substitute $q=\omega$:
\begin{equation*}
\begin{split}
([n]^i_{\omega})^{(j)} &= \sum_{a_1+\ldots+a_i = j} \binom{j}{a_1,\ldots,a_i} \prod_{1 \le k \le i} \frac{P_{n,a_k}(\omega)}{(\omega-1)^{a_k} \omega^{a_k}}\\
&= \frac{\sum_{a_1+\ldots+a_i = j} \binom{j}{a_1,\ldots,a_i} \prod_{1 \le k \le i} P_{n,a_k}(\omega)}{(\omega-1)^j\omega^j} = \frac{R_{n,i,j}(\omega)}{(\omega-1)^j \omega^j},
\end{split}
\end{equation*}
as needed. We now bound the degree of $R_{n,i,j}(t)$. By definition, $\deg P_{n,i} \le i-1$ if $i \ge 1$ and $P_{n,0}=0$, and so
\begin{equation*}
\deg R_{n,i,j} \le \max_{a_1+\ldots + a_i = j} \sum_{k=1}^{i} \deg P_{n,a_k} \le \max_{a_1+\ldots + a_i = j} \sum_{k=1}^{i} (a_k-1) = j-i, 
\end{equation*}
which in particular shows that $R_{n,i,j}=0$ if $j<i$. Finally, we compute $R_{n,i,i}$. We have just established that $R_{n,i,i}$ is a constant polynomial, and in particular $R_{n,i,i} = R_{n,i,i}(\omega_n)$. From the values $P_{n,1}(\omega_n)=n$ and $P_{n,0}(\omega_n)=0$, and from the fact that $a_1+\ldots+a_i = i$ implies that either $a_k=1$ for all $1 \le k \le i$ or $a_k = 0$ for some $k$, it follows that 
\begin{equation*}
\begin{split}
R_{n,i,i} &= R_{n,i,i}(\omega_n) = \sum_{a_1+\ldots+a_i = i} \binom{i}{a_1,\ldots,a_i} \prod_{1 \le k \le i} P_{n,a_k}(\omega_n) \\
&= \binom{i}{\underbrace{1,\ldots,1}_{i}} P_{n,1}(\omega_n)^i + \sum_{a_1+\ldots + a_i = i, \, a_k = 0 \text{ for some k}} \binom{i}{a_1,\ldots,a_i} \prod_{1 \le k \le i} P_{n,a_k}(\omega_n)\\
&= i!n^i, 
\end{split}
\end{equation*}
as needed.
\end{proof}
\begin{proposition}\label{prop:digits}
Let $f(q) \in \mathbb{C}[q]$, $n \ge 2$ and $r \ge 1$. Assume that for any $0 \le i \le r-1$, the function $g_i \colon \mu_n \to \mathbb{C}$, $\omega \mapsto \omega^i f^{(i)}(\omega)$ depends only on the order of $\omega$. Then the following hold.
\begin{enumerate}
\item For $0 \le i \le r-1$, define $f_i(q)$ recursively by
\begin{equation}\label{eq:propdigitform}
f_i(q) = \frac{1}{i!n^i} \Big((q-1)^i G_{h_i,n}(q)- \sum_{m_1=0}^{i-1} \sum_{m_2=m_1}^{i} \binom{i}{m_2} f_{m_1}^{(i-m_2)}(q)R_{n,m_1,m_2}(q)(q-1)^{i-m_2}q^{i-m_2}\Big),
\end{equation}
where
\begin{equation*}
h_i\colon D_n \to \mathbb{C}, \quad h_i(d) = 
(\omega_n^d)^i f^{(i)}(\omega_n^d)
\end{equation*}
and  $R_{n,m_1,m_2}(t) \in \mathbb{Z}[t]$ are defined in Lemma~\ref{lem:derivn}. 
Then for
\begin{equation}\label{eq:defr}
r(q) = \sum_{i=0}^{r-1} f_i(q) [n]_q^i
\end{equation}
we have 
\begin{equation}\label{eq:rderiv}
f^{(i)}(\omega) = r^{(i)}(\omega)
\end{equation}
for all $0 \le i \le r-1$ and $\omega \in \mu_n \setminus \{1\}$.
\item Let $p$ be the smallest prime divisor of $n$. For all $0 \le i\le r-1$ we have 
\begin{equation}\label{eq:degbound}
\deg f_i \le n-p+i.
\end{equation} 
\item For all $0 \le i\le \min \{ p-2,r-1\}$, 
the $i$-th $[n]_q$-digit of $f$ is $f_i$.
\item For all $1 \le i \le r-1$, $f_i(q)$ is a multiple of $q-1$.
\end{enumerate}
\end{proposition}
\begin{proof}
We prove \eqref{eq:rderiv}, \eqref{eq:degbound} by induction on $i$. For $i=0$, this is an application of Lemma~\ref{lemrem}. Suppose now that \eqref{eq:rderiv}, \eqref{eq:degbound} hold for all $i \le k-1$. To prove that \eqref{eq:rderiv} holds for $k$ in place of $i$ (assuming that $k \le r-1$), we note that the induction hypothesis implies that
\begin{equation*}
f(q) - \sum_{i=0}^{k-1} f_i(q) [n]_q^i = f(q)-r(q) + [n]_q^k \left(\sum_{i=k}^{r} f_i(q) [n]_q^{i-k}\right)
\end{equation*}
is divisible by $[n]_q^{k}$. By $k$ successive applications of L'H\^{o}pital's rule and by Lemma~\ref{lem:derivn}, we have for all $\omega \in \mu_{n} \setminus \{1\}$
\begin{equation}\label{eq:lopital}
\begin{split}
\lim_{q \to \omega}&\frac{ f(q) - \sum_{i=0}^{k-1} f_i(q) [n]_q^i}{[n]_q^{k}}= \lim_{q \to \omega}\frac{ f^{(k)}(q) - \sum_{i=0}^{k-1} (f_i(q) [n]_q^i)^{(k)}}{([n]_q^{k})^{(k)}}\\
&= \frac{ f^{(k)}(\omega) - \sum_{i=0}^{k-1} \sum_{j=0}^{k}\binom{k}{j} f_i^{(k-j)}(\omega) R_{n,i,j}(\omega)(\omega-1)^{-j}\omega^{-j} }{k!n^k/(\omega^k(\omega-1)^k)}\\
&=\frac{1}{k! n^k}  \Big( (\omega-1)^k\omega^kf^{(k)}(\omega) -  \sum_{i=0}^{k-1} \sum_{j=i}^{k} \binom{k}{j} f_i^{(k-j)}(\omega)R_{n,i,j}(\omega) (\omega-1)^{k-j} \omega^{k-j}\Big).
\end{split}
\end{equation}
By Lemma~\ref{lemrem}, $\omega^k f^{(k)}(\omega) = G_{h_k,n}(\omega)$ for all $\omega \in \mu_{n} \setminus \{1\}$, which together with \eqref{eq:lopital} shows that
\begin{equation*}
\lim_{q \to \omega}\frac{ f(q) - \sum_{i=0}^{k-1} f_i(q) [n]_q^i}{[n]_q^{k}} = f_k(\omega).
\end{equation*}
This shows that $(f(q) - \sum_{i=0}^{k} f_i(q) [n]_q^i)/([n]_q^k)$ vanishes on the roots of $[n]_q$, and so $f(q) - \sum_{i=0}^{k} f_i(q) [n]_q^i$ is divisible by $[n]_q^{k+1}$ , thus implying that \eqref{eq:defr} holds for $k$ in place of $i$.

To prove that \eqref{eq:degbound} holds for $k$ in place of $i$, note that $\deg f_k(q) \le \max \{ S_1,S_2\}$ where
\begin{equation*}
S_1 = \deg ((q-1)^k G_{h_k,n}(q))  \le  k+\max_{d \in D_n} \deg \frac{[n]_q}{[n/d]_q} = n-p+k
\end{equation*}
and
\begin{equation*}
\begin{split}
S_2 &= \deg \sum_{i=0}^{k-1} \sum_{j=i}^{k} \binom{k}{j} f_i^{(k-j)}(q)R_{n,i,j}(q) (q-1)^{k-j}q^{k-j} \\
&\le \max_{0 \le i \le k-1,\, i \le j \le k} (\deg f_i^{(k-j)} + \deg R_{n,i,j} + 2(k-j) )\\
&\le \max_{0 \le i \le k-1,\, i \le j \le k} (\deg(f_i)-(k-j) + j-i +  2(k-j))\\
& \le \max_{0 \le i \le k-1,\, i \le j \le k} (n-p+i+(k-i)) = n-p+k,
\end{split}
\end{equation*}
by Lemma~\ref{lem:derivn} and our inductive assumption on $\deg f_i$. Thus $\deg f_k(q) \le n-p+k$, as needed.

Let $\tilde{f}(q) = \sum_{i=0}^{\min\{p-2,r-1\}} f_i(q) [n]_q^i$. By \eqref{eq:degbound}, $\deg \tilde{f} <\deg [n]_q^{\min\{p-2,r-1\}+1}$. By \eqref{eq:defr}, $\tilde{f}(q) - f(q)$ is divisible by $[n]_q^{\min\{p-2,r-1\}+1}$. Thus, $\tilde{f}(q)$ is the remainder of $f(q)$ upon division by $[n]^{\min\{p-2,r-1\}+1}_q$, which proves that $f_i$ is the $i$-th $[n]_q$-digit of $f(q)$ for $0 \le i \le \min\{p-2,r-1\}$.

We turn to prove, by induction on $i$, that the $f_i$-s are divisible by $q-1$ when $1 \le i \le r-1$. For $i=1$, as $R_{n,0,1}=0$, \eqref{eq:propdigitform} shows that $f_1$ is a multiple of $(q-1)^1$ by construction. We now assume that $f_i$ is divisible by $q-1$ for all $1 \le i \le c$ for some $1 \le c<r-1$, and show that $f_{c+1}$ is also divisible by $q+1$. For $i=c+1$, all the summands in \eqref{eq:propdigitform} are multiples of $q-1$, except possibly the summands corresponding to $(m_1,m_2)$ with $m_2=c+1$, which look like $f_{m_1}(q) R_{n,m_1,c+1}(q)$. Note that we may assume that $m_1 \ge 1$ since $R_{n,0,c+1}=0$. Since $1 \le m_1<i$, we can use the induction hypothesis to deduce that these summands are also divisible by $q-1$, and so $f_{c+1}$ is divisible by $q-1$, as needed.
\end{proof}

\subsection{Proof of Theorem~\ref{thmexpansion}}\label{proofexpansion}
Formula \eqref{eq:digitsthmexp} follows from applying Proposition~\ref{prop:digits} with $f(q) = a_{nm}(q)$, which also tells us that the $f_i$-s are divisible by $q-1$ for $1 \le i\le \min\{p-2,r-1\}$. After substituting $q=1$ in 
\begin{equation*}
a_{nm}(q) \equiv \sum_{i=0}^{\min\{p-2,r-1\}} (q-1)\frac{f_i(q)}{q-1} [n]_q^i \bmod [n]_q^{1+\min\{p-2,r-1\}}
\end{equation*}
we obtain
\begin{equation}\label{eq:anm1}
\begin{split}
a_{nm}(1) &\equiv f_0(1) = G_{g_0}(1) =\sum_{d\in D_n} \sum_{e \mid d} \mu(\frac{d}{e}) a_{me}(1) \\
&=\sum_{e \mid n} a_{me}(1) \sum_{d: e\mid d \mid n, \, d \neq n}
\mu(\frac{d}{e})  \bmod n^{1+\min \{ p-2,r-1\}}.
\end{split}
\end{equation}
By Lemma~\ref{lem:mob}, the inner sum in \eqref{eq:anm1} is $1_{e=n}-\mu(\frac{n}{e})$, and so \eqref{eq:anm1} becomes \eqref{eq:supermu}, as needed.  \qed
\section{Examples (II)}
\subsection{Derivatives of $q$-binomial coefficients at roots of unity}\label{sec:derivativesofbinom}
\begin{lem}\label{lem:binomnplusk} \cite[Prop.~4.2]{reiner2004}
Let $n,k$ be non-negative integers. Let $\omega \in \mu_n$. We have
\begin{equation*}
{n+k \brack k}_{\omega} = \binom{\frac{n}{\mathrm{ord}(\omega)} + \lfloor \frac{k}{\mathrm{ord}(\omega)}\rfloor}{\lfloor \frac{k}{\mathrm{ord}(\omega)} \rfloor}.
\end{equation*}
\end{lem}
In the following propositions we use the convention that $\binom{a}{-1}=0$ for any integer $a$.
\begin{proposition}\label{propfirstderiv}
Let $n,k$ be integers with $n \ge k\ge  0$. Let $i \in \mathbb{Z}$ and set $d=(n,i)$. If $n \mid dk$, we have
\begin{equation*}
 \omega_n^i {n \brack k}^{'}_{\omega_n^i} = \binom{n}{2} \binom{d-1}{\frac{kd}{n}-1}
- \binom{k}{2} \binom{d}{\frac{kd}{n}} .
\end{equation*}
If $n \nmid dk$, we have
\begin{equation*}
 \omega_n^{i\binom{k}{2}+i}  {n \brack k}^{'}_{\omega_n^i} = \frac{n}{\omega_n^{ik}-1} (-1)^{k+\lfloor (k-1)d/n \rfloor + 1} \binom{d-1}{\lfloor (k-1)d/n \rfloor}.
\end{equation*}
\end{proposition}
\begin{proof}
We start by differentiating \eqref{qbinom} with respect to $q$:
\begin{equation}\label{qbinomdif}
\prod_{j=0}^{n-1} (1+tq^j) \cdot \sum_{j=0}^{n-1} \frac{jq^{j-1}t}{1+tq^j} = \sum_{r=0}^{n} q^{\binom{r}{2}-1} t^r \Big( \binom{r}{2} {n \brack r}_q +q  {n \brack r}^{'}_{q} \Big).
\end{equation}
We plug $q=\omega_n^i$ in \eqref{qbinomdif} and obtain
\begin{equation}\label{qbinomdifroot}
\prod_{j=0}^{n-1} (1+t\omega^{ij}_n) \cdot \sum_{j=0}^{n-1} \frac{j\omega_n^{i(j-1)}t}{1+t\omega_n^{ij}} = \sum_{r=0}^{n} \omega_n^{i(\binom{r}{2}-1)}  t^r \Big( \binom{r}{2} {n \brack r}_{\omega^i_n} +\omega_n^i {n \brack r}^{'}_{\omega_n^i} \Big).
\end{equation}
We can simplify the left-hand side of \eqref{qbinomdifroot} by using $\prod_{j=0}^{n-1} (1+t\omega^{ij}_n)=(1-(-t)^{n/d})^d$, and the right-hand side by using Lemma~\ref{lemroots}. We obtain
\begin{equation}\label{qbinomdifroot2}
(1-(-t)^{n/d})^d \cdot \sum_{j=0}^{n-1} \frac{j\omega_n^{i(j-1)}t}{1+t\omega_n^{ij}} = \sum_{r=0}^{n} \omega_n^{i(\binom{r}{2}-1)} t^r \Big( 1_{n \mid dr} \cdot \binom{r}{2} \binom{d}{dr/n} +\omega_n^i  {n \brack r}^{'}_{\omega_n^i} \Big).
\end{equation}
Writing $\frac{1}{1+t\omega_n^{ij}}$ as $\sum_{m \ge 0} t^m (-\omega_n^{ij})^m$, we compare the coefficients of $t^k$ on both sides of \eqref{qbinomdifroot2} and multiply the result by $\omega_n^i$:
\begin{equation}\label{qbinomdifroot3}
\sum_{0 \le s \le (k-1)d/n} \binom{d}{s} (-1)^{(\frac{n}{d}+1)s} \sum_{j=0}^{n-1} (-j) (-\omega_n^{ij})^{k-\frac{n}{d}s} = \omega_n^{i\binom{k}{2}} \Big(   1_{n \mid dk} \cdot \binom{k}{2} \binom{d}{dk/n} + \omega_n^i {n \brack k}^{'}_{\omega_n^i} \Big).
\end{equation}
We can simplify the left-hand side of \eqref{qbinomdifroot3} by noting that $\omega_n^{ij\frac{n}{d}}=1$, which leads to
\begin{equation}\label{qbinomdifroot4}
(-1)^{k+1}\sum_{0 \le s \le (k-1)d/n} \binom{d}{s} (-1)^s \sum_{j=0}^{n-1} j \omega_n^{ikj} = \omega_n^{i\binom{k}{2}} \Big(   1_{n \mid dk} \cdot \binom{k}{2} \binom{d}{dk/n} + \omega_n^i {n \brack k}^{'}_{\omega_n^i} \Big).
\end{equation}
The formal identity $\sum_{j=0}^{n-1} jx^j = x(\sum_{j=0}^{n-1} x^j)' =x(\frac{1-x^n}{1-x})' =  x \frac{(n-1)x^n - nx^{n-1} + 1}{(1-x)^2}$ shows that 
\begin{equation}\label{finiterootsum}
\sum_{j=0}^{n-1} j (\omega_n^m)^j = \begin{cases} \binom{n}{2} & \mbox{if $n \mid m$,}\\ \frac{n}{\omega_n^m-1} & \mbox{otherwise.} \end{cases}
\end{equation}
Applying \eqref{finiterootsum} with $m=ik$, we simplify \eqref{qbinomdifroot4} as follows. If $n \mid dk$, we have
\begin{equation}\label{nmidbeforesimp}
\binom{n}{2}(-1)^{k+1}\sum_{0 \le s \le (k-1)d/n} \binom{d}{s} (-1)^s  =
 \omega_n^{i\binom{k}{2}} \Big(   \binom{k}{2} \binom{d}{dk/n} + \omega_n^i {n \brack k}^{'}_{\omega_n^i} \Big).
\end{equation}
Otherwise, we have
\begin{equation}\label{notmidbeforesimp}
\frac{n}{\omega_n^{ik}-1} (-1)^{k+1} \sum_{0 \le s \le (k-1)d/n} \binom{d}{s} (-1)^s  =
 \omega_n^{i\binom{k}{2}+i}  {n \brack k}^{'}_{\omega_n^i}.
\end{equation}
The identity 
\begin{equation}\label{alterbinomsum}
\sum_{s=0}^{r}  \binom{d}{s} (-1)^s= (-1)^r \binom{d-1}{r},
\end{equation}
which may be proved comparing coefficients in $\frac{(1-x)^d}{1-x}=(1-x)^{d-1}$, together with the observation  that $\omega_n^{i\binom{k}{2}}=(-1)^{k+\frac{kd}{n}}$ when $n \mid dk$, allow us to simplify \eqref{nmidbeforesimp}, \eqref{notmidbeforesimp} and to obtain the result of the proposition.
\end{proof}

\begin{proposition}\label{propbinom1stderiv}
Let $n,k$ be integers with $n \ge k \ge 0$. Let $i \in \mathbb{Z}$ and set $d=(n,i)$. If $n \mid dk$, we have
\begin{equation}\label{eq:secondderiv}
\begin{split}
 \omega_n^{2i} {n \brack k}^{''}_{\omega_n^i} &= \left(n \frac{(3d^2+1)n^2-6d^2n+2d^2}{12d}-2\binom{n}{3} \right)  \left(n\binom{d-2}{\frac{kd}{n}-2} -k\binom{d-1}{\frac{kd}{n}-1}\right)\\
&-\binom{n}{2} \left(n\binom{d-2}{\frac{kd}{n}-2} -(k-1)\binom{d-1}{\frac{kd}{n}-1}\right) +\binom{n}{2}^2 \binom{d-1}{\frac{kd}{n}-1}\\
&- \left(\binom{k}{2}-1\right)\binom{k}{2} \binom{d}{\frac{kd}{n}} - k(k-1) \left(\binom{n}{2}\binom{d-1}{\frac{kd}{n}-1}-\binom{k}{2} \binom{d}{\frac{kd}{n}}\right).
\end{split}
\end{equation}
\end{proposition}
\begin{remark}
Although the expression in the right-hand side of \eqref{eq:secondderiv} can be simplified (see Corollary~\ref{cor:binomder}), as currently written it constitutes a proof that $\omega_n^{2i} {n \brack k}^{''}_{\omega_n^i} \in \mathbb{Z}$ when $n\mid ik$.
\end{remark}
\begin{proof}
We start by differentiating \eqref{qbinom} twice with respect to $q$, which is the same as differentiating \eqref{qbinomdif} once with respect to $q$, and the result is
\begin{equation}\label{qbinomdif2}
\begin{split}
\prod_{j=0}^{n-1} (1+tq^j) \cdot &\left( \sum_{j=0}^{n-1} \frac{ jq^{j-2}t( j-1 -q^jt )}{(1+tq^j)^2} + \left(\sum_{j=0}^{n-1} \frac{jq^{j-1}t}{1+tq^j} \right)^2 \right) \\
&= \sum_{r=0}^{n} q^{\binom{r}{2}-2} t^r \left( \left(\binom{r}{2}- 1\right)\binom{r}{2} {n \brack r}_q +q r(r-1)  {n \brack r}^{'}_{q}+q^2{n \brack r}^{''}_q \right).
\end{split}
\end{equation}
We plug $q=\omega_n^i$ in \eqref{qbinomdif2} and obtain 
\begin{equation}\label{qbinomdifroot2ndderiv}
\begin{split}
(1-(-t)^{n/d})^{d} \cdot &\left( \sum_{j=0}^{n-1} \frac{ j\omega_n^{i(j-2)}t( j-1 -\omega_n^{ij} t )}{(1+t \omega_n^{ij})^2} + \left(\sum_{j=0}^{n-1} \frac{j \omega_n^{i(j-1)}t}{1+t \omega_n^{ij}}\right)^2 \right) \\
&= \sum_{r=0}^{n} \omega_n^{i(\binom{r}{2}-2)} t^r \left( \left(\binom{r}{2}- 1\right)\binom{r}{2} {n \brack r}_{\omega_n^i} +\omega_n^i r(r-1)  {n \brack r}^{'}_{\omega_n^i}+\omega_n^{2i}{n \brack r}^{''}_{\omega_n^i} \right).
\end{split}
\end{equation}
Let 
\begin{equation*}
S_1(t,q) = \sum_{j=0}^{n-1} \frac{j q^{j-2} t(j-1)}{(1+tq^j)^2}, \quad S_2(t,q) = -\sum_{j=0}^{n-1} \frac{j q^{2j-2} t^2}{(1+tq^j)^2}, \quad S_3(t,q) = \Big( \sum_{j=0}^{n-1} \frac{j q^{j-1} t}{1+tq^j} \Big)^2.
\end{equation*}
Comparing the coefficient of $t^k$ in \eqref{qbinomdifroot2ndderiv}, multiplying the result by $\omega_n^{2i}$ and using Lemma~\ref{lemroots} and Proposition~\ref{propfirstderiv} to simplify it, we obtain
\begin{equation}\label{twosidescoeff}
\begin{split}
\omega_n^{2i}&\sum_{0 \le s \le \frac{dk}{n}-1} \binom{d}{s}(-1)^{(\frac{n}{d}+1)s} [t^{k-\frac{n}{d}s}] (S_1(t,\omega_n^i) + S_2(t,\omega_n^i) + S_3(t,\omega_n^i)) \\
& =  (-1)^{k+\frac{kd}{n}} \left( \left(\binom{k}{2}-1\right)\binom{k}{2} \binom{d}{\frac{kd}{n}} + k(k-1) \left(\binom{n}{2}\binom{d-1}{\frac{kd}{n}-1}-\binom{k}{2} \binom{d}{\frac{kd}{n}}\right) + \omega_n^{2i} {n \brack k}^{''}_{\omega_n^i}\right).
\end{split}
\end{equation}
We now compute the coefficient of $t^r$ in $S_i(t,\omega_n^i)$, assuming that $
\frac{n}{d} \mid r$, $r \ge 1$. We begin with $S_1(t,q)$. Since $\frac{1}{(1+t)^2} = \sum_{r \ge 0} (r+1)(-t)^r$ in $\mathbb{C}[[t]]$, we have
\begin{equation}\label{s1eval}
\begin{split}
S_1(t,q) &= -\frac{1}{q^2} \sum_{j=0}^{n-1} j(j-1) \sum_{r \ge 1} r (-tq^j)^{r}\\
&= -\frac{1}{q^2} \sum_{r \ge 1} r(-t)^r \sum_{j=0}^{n-1} j(j-1) q^{jr} \implies \\
\omega_n^{2i}[t^r]S_1(t,\omega_n^i ) &=  - r(-1)^r \sum_{j=0}^{n-1} j(j-1) = -2\binom{n}{3} r(-1)^r . 
\end{split}
\end{equation}
We proceed with $S_2(t,q)$:
\begin{equation}\label{s2eval}
\begin{split}
S_2(t,q) &= -\frac{1}{q^2} \sum_{j=0}^{n-1} j \sum_{r \ge 1} (r-1) (-tq^j)^{r} \\
&= -\frac{1}{q^2} \sum_{r \ge 1} (r-1)(-t)^r \sum_{j=0}^{n-1} j q^{jr} \implies \\
\omega_n^{2i}[t^r]S_2(t,\omega_n^i ) &=  -(r-1)(-1)^r \sum_{j=0}^{n-1} j = -\binom{n}{2} (r-1)(-1)^r . 
\end{split}
\end{equation}
We now treat $S_3(t,q)$.
\begin{equation}\label{s3evalpre}
\begin{split}
S_3(t,q) &= \Big(\sum_{j=0}^{n-1} jq^{j-1} t \sum_{m \ge 0} (-tq^j)^m \Big)^2 \\
&= \frac{1}{q^2} \sum_{r \ge 2} (-t)^r \sum_{\substack{r_1+r_2 = r \\ r_i \ge 1}} \sum_{0 \le j_1,j_2 \le n-1} j_1 j_2 q^{j_1r_1 + j_2 r_2}\implies \\
\omega_n^{2i}[t^r]S_3(t,\omega_n^i ) &=  (-1)^r \sum_{0 \le j_1,j_2 \le n-1} j_1 j_2 \sum_{r_1=1}^{r-1} \omega_n^{ir_1(j_1-j_2)} =(-1)^r \sum_{0 \le j_1,j_2 \le n-1} j_1 j_2 (-1+r \cdot 1_{j_1 \equiv j_2 \bmod \frac{n}{d}}) \\
&= (-1)^r \Big( -\binom{n}{2}^2 +r \sum_{0 \le j_1,j_2 \le n-1, j_1 \equiv j_2 \bmod \frac{n}{d}} 1 \Big).
\end{split}
\end{equation}
Since 
\begin{equation*}
\begin{split}
\sum_{0 \le j_1,j_2 \le n-1, j_1 \equiv j_2 \bmod \frac{n}{d}} 1 &= \sum_{m=0}^{\frac{n}{d}-1} \Big(\sum_{0 \le j \le n-1, j \equiv m \bmod \frac{n}{d}} 1\Big)^2 = \sum_{m=0}^{\frac{n}{d}-1} \Big(\sum_{s=0}^{d-1} (m+\frac{n}{d} s)\Big)^2 \\
&= \sum_{m=0}^{\frac{n}{d}-1} \Big(dm + \frac{n}{d} \binom{d}{2} \Big)^2=n\frac{(3d^2+1)n^2-6d^2n+2d^2}{12d},
\end{split}
\end{equation*}
we obtain from \eqref{s3evalpre} that
\begin{equation}\label{s3eval}
\omega_n^{2i}[t^r]S_3(t,\omega_n^i ) =(-1)^r \Big( -\binom{n}{2}^2 +r n  \frac{(3d^2+1)n^2-6d^2n+2d^2}{12d}\Big).
\end{equation}
Using \eqref{s1eval}, \eqref{s2eval} and \eqref{s3eval}, the left-hand side of \eqref{twosidescoeff} becomes
\begin{equation}\label{lhsderiv2simp1}
\begin{split}
(-1)^k&\Big(n \frac{(3d^2+1)n^2-6d^2n+2d^2}{12d}-2\binom{n}{3} \Big)\sum_{0 \le s \le \frac{kd}{n}-1} \binom{d}{s}(-1)^{s} 
 (k-\frac{n}{d}s)\\
&-(-1)^k\binom{n}{2}\sum_{0 \le s \le \frac{kd}{n}-1} \binom{d}{s}
 (k-\frac{n}{d}s - 1)(-1)^{s}\\
 &-(-1)^k\binom{n}{2}^2
 \sum_{0 \le s \le \frac{kd}{n}-1} \binom{d}{s} (-1)^{s}.
  \end{split}
\end{equation}
Using \eqref{alterbinomsum} and its variant 
\begin{equation*}
\sum_{s=0}^{r} (-1)^s \binom{d}{s}s  = 
-d\sum_{s=1}^{r} (-1)^{s-1} \binom{d-1}{s-1} = d(-1)^r \binom{d-2}{r-1},
\end{equation*}
we can simplify \eqref{lhsderiv2simp1} as
\begin{equation}\label{lhsderiv2simp2}
\begin{split}
(-1)^{k+\frac{kd}{n}}&\left( \left(n \frac{(3d^2+1)n^2-6d^2n+2d^2}{12d}-2\binom{n}{3} \right)  \left(n\binom{d-2}{\frac{kd}{n}-2} -k\binom{d-1}{\frac{kd}{n}-1}\right) \right.\\
&\quad \left. -\binom{n}{2} \left(n\binom{d-2}{\frac{kd}{n}-2} -(k-1)\binom{d-1}{\frac{kd}{n}-1}\right) +\binom{n}{2}^2 \binom{d-1}{\frac{kd}{n}-1}\right).
  \end{split}
\end{equation}
Replacing the left-hand side of \eqref{twosidescoeff} with \eqref{lhsderiv2simp2}, dividing by $(-1)^{k+\frac{kd}{n}}$ and isolating the term $\omega_n^{2i} {n \brack k}_{\omega_n^i}$, we conclude the proof.
\end{proof}
We can simplify the expressions in Propositions~\ref{propfirstderiv} and \ref{propbinom1stderiv} using the relations $\binom{d-1}{\frac{kd}{n}-1} = \binom{d}{\frac{kd}{n}} \frac{k}{n}$ for $\frac{kd}{n}\ge 1$ and $\binom{d-2}{\frac{kd}{n}-2} = \binom{d-1}{\frac{kd}{n}-1}\frac{\frac{kd}{n}-1}{d-1}$ for $\frac{kd}{n} \ge 2$, and obtain the following corollary from Lemma~\ref{lemroots} and these propositions.
\begin{cor}\label{cor:binomder}
Fix $n \ge k \ge 0$. For every $0 \le j \le 2$, the function $\omega^j {n \brack k}_{\omega}^{(j)}$, as a function of $\omega\in \mu_{\gcd(n,k)}$, depends only on the order of $\omega$ and assumes integer values. In fact, for any primitive root of unity $\omega \in \mu_{\gcd(n,k)}$ of order $d$, we have
\begin{equation*}
\begin{split}
{n \brack k}_{\omega} &=  \binom{n/d}{k/d},\\
\omega {n \brack k}'_{\omega} &=  \binom{n/d}{k/d} \frac{k(n-k)}{2},\\
\omega^2 {n \brack k}''_{\omega} &=  \binom{n/d}{k/d} k(n-k) \left(\frac{k(n-k)}{4}+ \frac{nd-5}{12}\right).
\end{split}
\end{equation*}
If $\omega \in \mu_{n} \setminus \mu_{k}$ is a primitive root of unity of order $d$, we have
\begin{equation*}
\begin{split}
{n \brack k}_{\omega} & = 0, \\
 \omega^{\binom{k}{2}+1}  {n \brack k}^{'}_{\omega} &= \frac{n}{\omega^{k}-1} (-1)^{k+\lfloor (k-1)/d \rfloor + 1} \binom{\frac{n}{d}-1}{\lfloor (k-1)/d \rfloor}.
 \end{split}
\end{equation*}
\end{cor}
\subsection{Sums of roots of unity}
The following lemma was essentially proved by Shi and Pan \cite{shi2007}. We provide a different proof.
\begin{lem}\label{lem:rootsunitysum}
Let $n \ge 1$. Let $\omega$ be a primitive root of unity of order $n$. Then
\begin{equation*}
\begin{split}
\sum_{i=1}^{n-1} \frac{1}{\omega^i - 1} &= -\frac{n-1}{2},\\
\sum_{i=1}^{n-1} \frac{1}{(\omega^i - 1)^2} &= -\frac{(n-1)(n-5)}{12}.
\end{split}
\end{equation*}
\end{lem}
\begin{proof}
Substituting $q+1$ in place of $q$ in $\prod_{i=1}^{n-1} (q-\omega^i) = \frac{q^n-1}{q-1}$, we obtain
\begin{equation}\label{eq:prodshifted}
\prod_{i=1}^{n-1} (q+1-\omega^i) = \frac{(q+1)^n-1}{q}=q^{n-1}+\ldots + \binom{n}{3}q^2 + \binom{n}{2}q + n.
\end{equation}
By equating coefficients in \eqref{eq:prodshifted}, we see that 
\begin{equation*}
\begin{split}
\sigma_{n-1}&:=\prod_{i=1}^{n-1}(\omega^i-1) = n(-1)^{n-1}, \\
\sigma_{n-2} &:= \sum_{i=1}^{n-1} \prod_{1 \le j \le n-1,\, j \neq i} (\omega^j-1) = \binom{n}{2}(-1)^{n-2},\\
\sigma_{n-3} &:= \sum_{1 \le i<j \le n-1} \prod_{1 \le k \le n-1,\, k \neq i,j} (\omega^k-1) = \binom{n}{3}(-1)^{n-3}.
\end{split}
\end{equation*}
Thus,
\begin{equation*}
\sum_{i=1}^{n-1} \frac{1}{\omega^i-1} = \frac{\sigma_{n-2}}{\sigma_{n-1}} = -\frac{n-1}{2}
\end{equation*}
and
\begin{equation*}
\sum_{i=1}^{n-1} \frac{1}{(\omega^i-1)^2} = \left(\frac{\sigma_{n-2}}{\sigma_{n-1}}\right)^2 -2 \frac{\sigma_{n-3}}{\sigma_{n-1}} = \left(-\frac{n-1}{2}\right)^2 - 2 \frac{\binom{n}{3}}{n}= -\frac{(n-1)(n-5)}{12},
\end{equation*}
as needed.
\end{proof}
\subsection{Proof of Theorem~\ref{thm:superqbinom}}
Let $a \ge b \ge 1$ and define $a_n(q) = {an \brack bn}_q$. Corollary~\ref{cor:binomder} implies that \eqref{eq:an1b}--\eqref{eq:an3b} hold for any $n \ge 1$ and $\omega \in \mu_n$. \qed 

\subsection{Proof of Theorem~\ref{thm:aprey}}\label{sec:superqapery}
We show that \eqref{eq:an1} holds by using Lemma~\ref{lemroots}, which tells us in particular that ${n \brack k}_{\omega}$ vanishes for $\omega \in \mu_n \setminus \mu_k$. For any $\omega \in \mu_n$,
\begin{equation*}
\begin{split}
a_n(\omega)&=\sum_{k=0}^{n} {n \brack k}_{\omega}^2 {n+k \brack k}_{\omega}^2 \omega^{f(n,k)} \\
&= \sum_{0 \le k \le n, \, \mathrm{ord}(\omega) \mid k} \binom{n/\mathrm{ord}(\omega)}{k/\mathrm{ord}(\omega)}^2 {n+k \brack k}_{\omega}^2 \omega^{f(n,k)}\\
&=\sum_{0 \le k \le n, \, \mathrm{ord}(\omega) \mid k} \binom{n/\mathrm{ord}(\omega)}{k/\mathrm{ord}(\omega)}^2 \binom{(n+k)/\mathrm{ord}(\omega)}{k/\mathrm{ord}(\omega)}^2= a_{n/\mathrm{ord}(\omega)}(1).
\end{split}
\end{equation*}
We show that \eqref{eq:an2} holds. The derivative of $a_n(q)$, times $q$, is given by
\begin{equation*}
qa_n'(q) = \sum_{k, \ell} {n \brack k}_q {n+k \brack k}_q q^{f(n,k)} \left(2q{n \brack k}'_q {n+k \brack k}_q + 2q{n \brack q}_q {n+k \brack k}'_q + f(n,k) {n \brack k}_q {n+k \brack k}_q\right),
\end{equation*}
and Corollary~\ref{cor:binomder} allows us to evaluate it at $q=\omega \in \mu_n$:
\begin{equation*}
\begin{split}
\omega a_n'(\omega) &= \sum_{0\le k \le n, \, \mathrm{ord}(\omega) \mid k} \binom{\frac{n}{\mathrm{ord}(\omega)}}{\frac{k}{\mathrm{ord}(\omega)}} \binom{\frac{n+k}{\mathrm{ord}(\omega)}}{\frac{k}{\mathrm{ord}(\omega)}}\\
&  \qquad \qquad \cdot \left(2\omega {n \brack k}'_{\omega} \binom{\frac{n+k}{\mathrm{ord}(\omega)}}{\frac{k}{\mathrm{ord}(\omega)}} + 2\omega \binom{\frac{n}{\mathrm{ord}(\omega)}}{\frac{k}{\mathrm{ord}(\omega)}} {n+k \brack k}'_{\omega}  +f(n,k)  \binom{\frac{n}{\mathrm{ord}(\omega)}}{\frac{k}{\mathrm{ord}(\omega)}} \binom{\frac{n+k}{\mathrm{ord}(\omega)}}{\frac{k}{\mathrm{ord}(\omega)}}\right)\\
&=\sum_{0\le k \le n, \, \mathrm{ord}(\omega) \mid k} \binom{\frac{n}{\mathrm{ord}(\omega)}}{\frac{k}{\mathrm{ord}(\omega)}}^2 \binom{\frac{n+k}{\mathrm{ord}(\omega)}}{\frac{k}{\mathrm{ord}(\omega)}}^2 (  k(n-k)+  kn  +f(n,k))\\
&= \mathrm{ord}(\omega)^2 \sum_{0\le k \le n, \, \mathrm{ord}(\omega) \mid k} \binom{\frac{n}{\mathrm{ord}(\omega)}}{\frac{k}{\mathrm{ord}(\omega)}}^2 \binom{\frac{n+k}{\mathrm{ord}(\omega)}}{\frac{k}{\mathrm{ord}(\omega)}}^2 \left(2\frac{k}{\mathrm{ord}(\omega)}\frac{n}{\mathrm{ord}(\omega)}-(\frac{k}{\mathrm{ord}(\omega)})^2 \right.\\
&  \qquad \left. + f(\frac{n}{\mathrm{ord}(\omega)},\frac{k}{\mathrm{ord}(\omega)})\right)= \mathrm{ord}(\omega)^2 a_{\frac{n}{\mathrm{ord}(\omega)},2xy-y^2+f(x,y)}.
\end{split}
\end{equation*}
We show that \eqref{eq:an3} holds. The second derivative of $a_n(q)$, times $q^2$, is given by
\begin{equation*}
q^2a_n''(q) = S_{n,1}(q) + S_{n,2}(q),
\end{equation*}
where 
\begin{equation*}
\begin{split}
S_{n,1}(q) &= \sum_{k=0}^{n} {n \brack k}_q q^{f(n,k)}\Big(2{n \brack k}''_q {n+k \brack k}^2_q q^2 + 8{n \brack k}'_q {n+k \brack k}_q {n+k \brack k}'_q q^2 + 4f(n,k){n \brack k}'_q {n+k \brack k}^2_q  q \\
&\qquad +2{n \brack k}_q ({n+k \brack k}'_q)^2 q^2   + 2 {n \brack k}_q {n+k \brack k}_q {n+k \brack k}''_q q^2 + 4f(n,k) {n \brack k}_q {n+k \brack k}_q {n+k \brack k}'_q q  \\
& \qquad + f(n,k) (f(n,k)-1){n \brack k}_q {n+k \brack k}^2_q \Big) \\
\end{split}
\end{equation*}
and 
\begin{equation}
S_{n,2}(q) = 2\sum_{k=0}^{n} q^2\Big({n \brack k}_q'\Big)^2 {n+k \brack k}_q^2 q^{f(n,k)}.
\end{equation}
Since ${n \brack k}_q$ vanishes on $\mu_n \setminus \mu_k$, Corollary~\ref{cor:binomder} allows us to evaluate $S_{n,1}(q)$ at $q=\omega \in \mu_n$ similarly to the evaluation of $a'_n(\omega)$ and $a_n(\omega)$:
\begin{equation}\label{eq:sn1val}
\begin{split}
S_{n,1}(\omega)&=\sum_{0 \le k \le n,\, \mathrm{ord}(\omega) \mid k} \binom{\frac{n}{\mathrm{ord}(\omega)}}{\frac{k}{\mathrm{ord}(\omega)}}^2 \binom{\frac{n+k}{\mathrm{ord}(\omega)}}{\frac{k}{\mathrm{ord}(\omega)}}^2  \bigg(2  k(n-k)\left(\frac{k(n-k)}{4}+\frac{n\cdot \mathrm{ord}(\omega)-5}{12}\right) + 8 \frac{k(n-k)}{2} \frac{kn}{2} \\
& \qquad  \quad + 4f(n,k) \frac{k(n-k)}{2} + 2 \left(\frac{kn}{2}\right)^2 +  2 kn \left(\frac{kn}{4} + \frac{(n+k)\cdot \mathrm{ord}(\omega)-5}{12}\right) \\
& \qquad \quad + 4f(n,k) \frac{kn}{2} + f(n,k)(f(n,k)-1)\bigg)\\
&= \sum_{0 \le k \le n,\, \mathrm{ord}(\omega) \mid k} \binom{\frac{n}{\mathrm{ord}(\omega)}}{\frac{k}{\mathrm{ord}(\omega)}}^2 \binom{\frac{n+k}{\mathrm{ord}(\omega)}}{\frac{k}{\mathrm{ord}(\omega)}}^2 (3.5n^2k^2 +0.5k^4-3nk^3  - \frac{5}{6}k(2n-k) + 2f(n,k)k(2n-k)\\
&\qquad \quad +f(n,k)(f(n,k)-1)+ \frac{n^2k}{3}\mathrm{ord}(\omega))\\
&=\mathrm{ord}(\omega)^4 a_{\frac{n}{\mathrm{ord}(\omega)}, 3.5x^2y^2+0.5y^4-3xy^3+\frac{x^2y}{3}+2f(x,y)y(2x-y)+f(x,y)^2} + \mathrm{ord}(\omega)^2 a_{\frac{n}{\mathrm{ord}(\omega)}, -\frac{5}{6}y(2x-y)-f(x,y)}.
\end{split}
\end{equation}
We turn to evaluate $S_{n,2}(q)$ at $q=\omega\in \mu_n$. We separate $S_{n,2}(q)$ into two sums -- one with the summands corresponding to $k$ divisible by $\mathrm{ord}(\omega)$, and another with the rest:
\begin{equation*}
S_{n,2}(q) = T_1(q) + T_2(q),
\end{equation*}
where
\begin{align*}
T_1(q) &= 2\sum_{0 \le k \le n,\, \mathrm{ord}(\omega) \mid k} q^2({n \brack k}_q')^2 {n+k \brack k}_{\omega}^2 q^{f(n,k)},\\
T_2(q) &= 2\sum_{0 \le k \le n,\, \mathrm{ord}(\omega) \nmid k} q^2({n \brack k}_q')^2 {n+k \brack k}_{\omega}^2 q^{f(n,k)}. 
\end{align*}
We use Lemma~\ref{lem:binomnplusk} and Corollary~\ref{cor:binomder}
to evaluate $T_1(\omega)$:
\begin{equation}\label{t1val}
\begin{split}
T_1(\omega) &= 2\sum_{0 \le k \le n,\, \mathrm{ord}(\omega) \mid k} \Big(\binom{\frac{n}{\mathrm{ord}(\omega)}}{\frac{k}{\mathrm{ord}(\omega)}} \frac{k(n-k)}{2} \Big)^2 \binom{\frac{n}{\mathrm{ord}(\omega)} + \frac{k}{\mathrm{ord}(\omega)}}{\frac{k}{\mathrm{ord}(\omega)}}^2\\
&=\mathrm{ord}(\omega)^4 a_{\frac{n}{\mathrm{ord}(\omega)},\frac{y^2(x-y)^2}{2}}.
\end{split}
\end{equation}
We turn to $T_2(\omega)$. By Lemma~\ref{lem:binomnplusk} and Corollary~\ref{cor:binomder}, we may evaluate $T_2(\omega)$ as follows:
\begin{equation}\label{eq:t2omega1}
T_2(\omega) = 2\sum_{0 \le k \le n,\, \mathrm{ord}(\omega) \nmid k} \left(\omega^{-\binom{k}{2}}\frac{n}{\omega^k-1} \binom{\frac{n}{\mathrm{ord}(\omega)}-1}{\lfloor \frac{k-1}{\mathrm{ord}(\omega)} \rfloor}\right)^2 \binom{\frac{n}{\mathrm{ord}(\omega)} + \lfloor \frac{k}{\mathrm{ord}(\omega)} \rfloor}{\lfloor \frac{k}{\mathrm{ord}(\omega)} \rfloor}^2\omega^{f(n,k)} .
\end{equation}
We write every $0 \le k \le n$ with $\mathrm{ord}(\omega) \nmid k$ as $k=\mathrm{ord}(\omega)j+i$ with $1 \le i \le \mathrm{ord}(\omega)-1$ and $0 \le j \le \frac{n}{\mathrm{ord}(\omega)}-1$. Now we have $\omega^{k}=\omega^i$, $\omega^{f(n,k)}=\omega^{k^2} = \omega^{i^2}$, $\lfloor \frac{k}{\mathrm{ord}(\omega)} \rfloor = \lfloor \frac{k-1}{\mathrm{ord}(\omega)} \rfloor = j$. We express \eqref{eq:t2omega1} as
\begin{equation}\label{eq:t2omega2}
\begin{split}
T_2(\omega) &= 2n^2\sum_{i=1}^{\mathrm{ord}(\omega)-1} \sum_{j=0}^{\frac{n}{\mathrm{ord}(\omega)}-1} 
\frac{\omega^{-i^2+i}}{(\omega^i-1)^2} \binom{\frac{n}{\mathrm{ord}(\omega)}-1}{j}^2 \binom{\frac{n}{\mathrm{ord}(\omega)}+j}{j}^2 \omega^{i^2}\\
&=2\left(\sum_{i=1}^{\mathrm{ord}(\omega)-1} 
\frac{\omega^{i}}{(\omega^i-1)^2}\right)(\sum_{j=0}^{\frac{n}{\mathrm{ord}(\omega)}-1}  \binom{\frac{n}{\mathrm{ord}(\omega)}}{j}^2 \binom{\frac{n}{\mathrm{ord}(\omega)}+j}{j}^2 (n-j\mathrm{ord}(\omega))^2 ).
\end{split}
\end{equation}
The sum $\sum_{i=1}^{\mathrm{ord}(\omega)-1} 
\frac{\omega^{i}}{(\omega^i-1)^2}$ is evaluated in Lemma~\ref{lem:rootsunitysum}. From \eqref{t1val}--\eqref{eq:t2omega2} and Lemma~\ref{lem:rootsunitysum}, we obtain
\begin{equation}\label{eq:sn2eval}
S_{n,2}(\omega)=\mathrm{ord}(\omega)^4 a_{\frac{n}{\mathrm{ord}(\omega)},\frac{y^2(x-y)^2}{2}-\frac{(x-y)^2}{6}}+ \mathrm{ord}(\omega)^2 a_{\frac{n}{\mathrm{ord}(\omega)},\frac{(x-y)^2}{6}}.
\end{equation}
From \eqref{eq:sn1val} and \eqref{eq:sn2eval}, we obtain \eqref{eq:an3}, as needed. \qed
\subsection{Proof of Theorem~\ref{thm:zeta}}\label{sec:superqzeta}
We show that \eqref{eq:an1z} holds by using Lemma~\ref{lemroots}, which tells us in particular that ${n \brack m}_{\omega}$ vanishes for $\omega \in \mu_n \setminus \mu_m$. For any $\omega \in \mu_n$,
\begin{equation*}
\begin{split}
a_n(\omega)&= \sum_{k,\ell} {n \brack k}_{\omega}^2 {n \brack \ell}_{\omega} {k \brack \ell}_{\omega} {k+\ell \brack n}_{\omega} \omega^{f(n,k,\ell)}  \\
&= \sum_{\mathrm{ord}(\omega) \mid k,\ell} \binom{\frac{n}{\mathrm{ord}(\omega)}}{\frac{k}{\mathrm{ord}(\omega)}}^2  \binom{\frac{n}{\mathrm{ord}(\omega)}}{\frac{\ell}{\mathrm{ord}(\omega)}} \binom{\frac{k}{\mathrm{ord}(\omega)}}{\frac{\ell}{\mathrm{ord}(\omega)}}  \binom{\frac{k+\ell}{\mathrm{ord}(\omega)}}{\frac{n}{\mathrm{ord}(\omega)}} \omega^{f(n,k,\ell)} =a_{n/\mathrm{ord}(\omega)}(1).
\end{split}
\end{equation*}
We show that \eqref{eq:an2z} holds. The derivative of $a_n(q)$, times $q$, is given by
\begin{equation}\label{eq:qderivzeta}
\begin{split}
qa_n'(q) &= \sum_{k,\ell} {n \brack k}_q q^{f(n,k,\ell)} \Big( {n \brack k}_q \big({n \brack \ell}_q {k \brack \ell}_q {k+\ell \brack n}_q\big)' \\
& \qquad + \big({n \brack k}_q f(n,k,\ell) + 2{n \brack k}_q' q\big) {n \brack \ell}_q {k \brack \ell}_q {k+\ell \brack n}_q \Big).
\end{split}
\end{equation}
After substituting $q =\omega \in \mu_n$ in \eqref{eq:qderivzeta}, we claim that only the terms with $\mathrm{ord}(\omega) \mid k,\ell$ contribute. Indeed, since each summand contains ${n \brack k}_q$, it follows by Lemma~\ref{lemroots} that each summand vanishes on $\omega$ if $\mathrm{ord}(\omega) \nmid k$. Similarly, since each summand contains either ${n \brack \ell}_q$ or ${k \brack \ell}_q$, Lemma~\ref{lemroots} again implies that each summand vanishes on $\omega$ if $\mathrm{ord}(\omega) \nmid \ell$. Thus, Corollary~\ref{cor:binomder} allows us to evaluate \eqref{eq:qderivzeta} at $q=\omega \in \mu_n$ as follows:
\begin{equation*}
\begin{split}
\omega a_n'(\omega) &= \sum_{\mathrm{ord}(\omega) \mid k,\ell} \binom{\frac{n}{\mathrm{ord}(\omega)}}{\frac{k}{\mathrm{ord}(\omega)}}^2 \binom{\frac{n}{\mathrm{ord}(\omega)}}{\frac{\ell}{\mathrm{ord}(\omega)}} \binom{\frac{k}{\mathrm{ord}(\omega)}}{\frac{\ell}{\mathrm{ord}(\omega)}} \binom{\frac{k+\ell}{\mathrm{ord}(\omega)}}{\frac{n}{\mathrm{ord}(\omega)}} \bigg( \frac{\ell(n-\ell)+\ell(k-\ell)+n(k+\ell-n)}{2} \\
& \qquad \quad  + f(n,k,\ell) + k(n-k)\bigg)= \mathrm{ord}(\omega)^2 a_{\frac{n}{\mathrm{ord}(\omega)},xz-y^2-z^2+\frac{3xy+yz-x^2}{2}+f(x,y,z)}.
\end{split}
\end{equation*}
Thus, \eqref{eq:an2z} holds. We show that \eqref{eq:an3z} holds. The second derivative of $a_n(q)$, times $q^2$, is given by
\begin{equation*}
q^2a_n''(q) = S_{n,1}(q) + S_{n,2}(q) + S_{n,3}(q),
\end{equation*}
where 
\begin{equation*}
\begin{split}
 S_{n,1}(q) &= \sum_{k, \ell} {n \brack k}_q q^{f(n,k,\ell)} \\
 & \qquad \cdot \bigg( \Big({n \brack \ell}_q {k \brack \ell}_q {k+\ell \brack n}_q \Big) \Big(2q^2 {n \brack k}''_q  + {n \brack k}_q  f(n,k,\ell)(f(n,k,\ell)-1) + 4q {n \brack k}'_q f(n,k,\ell)\Big)\\
 & \qquad \quad + {n \brack k}_q q^2 \Big( {n \brack \ell}''_q {k \brack \ell}_q {k+\ell \brack n}_q + {n \brack \ell}_q {k \brack \ell}''_q {k+\ell \brack n}_q + {n \brack \ell}_q {k \brack \ell}_q {k+\ell \brack n}''_q\\
 & \qquad \quad + 2{n \brack \ell}'_q {k \brack \ell}_q {k+\ell \brack n}'_q   + 2{n \brack \ell}_q {k \brack \ell}'_q {k+\ell \brack n}'_q\Big) \\
 &\qquad  \quad+ \Big(2{n \brack k}_q f(n,k,\ell)q  + 4{n \brack k}'_q q^2\Big) \Big(
{n \brack \ell}_q {k \brack \ell}_q {k+\ell \brack n}_q \Big)' \bigg), \\
S_{n,2}(q) &= 2\sum_{k,\ell} {n \brack k}^2_q q^{f(n,k,\ell)}q^2  {n \brack \ell}'_q {k \brack \ell}'_q {k+\ell \brack n}_q, \, \quad S_{n,3}(q) = 2\sum_{k,\ell} \big( q {n \brack k}'_q\big)^2 q^{f(n,k,\ell)} {n \brack \ell}_q {k \brack \ell}_q {k+\ell \brack n}_q.
\end{split}
\end{equation*}
Let
\begin{equation*}
\begin{split}
P_1(x,y,z) &= \frac{1}{4}(x^4 - 6x^3 y + 11x^2y^2 - 8xy^3 + 2y^4 - 4x^3z + 10x^2yz -2xy^2 z-4y^3z + 8x^2z^2  \\
& \qquad - 10xyz^2 + 9y^2z^2- 6xz^3- 2yz^3+2z^4)+f\cdot (-x^2+3xy-2y^2+2xz+yz-2z^2)\\
&\qquad +f^2 +\frac{x^2y - xy^2 + 2xyz + y^2z - yz^2}{12}
\end{split}
\end{equation*}
and 
\begin{equation*}
P_2(x,y,z) = \frac{5}{12}(x^2 -3xy+2y^2-2xz-yz+2z^2)-f.
\end{equation*}
Corollary~\ref{cor:binomder} allows us to evaluate $S_{n,1}(q)$ at $q=\omega \in \mu_n$ similarly to the evaluation of $a'_n(\omega)$ and $a_n(\omega)$:
\begin{equation}\label{eq:sn1valz}
\begin{split}
S_{n,1}(\omega)&=\sum_{k, \ell \mid \mathrm{ord}(\omega)} \binom{\frac{n}{\mathrm{ord}(\omega)}}{\frac{k}{\mathrm{ord}(\omega)}}^2 \binom{\frac{n}{\mathrm{ord}(\omega)}}{\frac{\ell}{\mathrm{ord}(\omega)}} \binom{\frac{k}{\mathrm{ord}(\omega)}}{\frac{\ell}{\mathrm{ord}(\omega)}} \binom{\frac{k+\ell}{\mathrm{ord}(\omega)}}{\frac{n}{\mathrm{ord}(\omega)}} \bigg( k(n-k)\big(\frac{k(n-k)}{2}+\frac{n \cdot \mathrm{ord}(\omega)-5}{6}\big) \\
& \qquad + f(n,k,\ell) (f(n,k,\ell)-1) + 2k(n-k)f(n,k,\ell) + \ell(n-\ell)\big(\frac{\ell(n-\ell)}{4}+\frac{n \cdot \mathrm{ord}(\omega)-5}{12}\big) \\
& \qquad +\ell(k-\ell)\big(\frac{\ell(k-\ell)}{4}+\frac{k \cdot \mathrm{ord}(\omega)-5}{12}\big) +
n(k+\ell-n)\big(\frac{n(k+\ell-n)}{4}+\frac{(k+\ell) \cdot \mathrm{ord}(\omega)-5}{12}\big) \\
& \qquad + \frac{\ell(n-\ell)n(k+\ell-n)+\ell(k-\ell)n(k+\ell-n)}{2} \\
& \qquad + \big(f(n,k,\ell) + k(n-k)\big)\big(\ell(n-\ell)+\ell(k-\ell)+n(k+\ell-n)\big)\bigg)\\
&=\mathrm{ord}(\omega)^4 a_{\frac{n}{\mathrm{ord}(\omega)}, P_1}+ \mathrm{ord}(\omega)^2 a_{\frac{n}{\mathrm{ord}(\omega)}, P_2}.
\end{split}
\end{equation}
We turn to evaluate $S_{n,2}(q)$ at $q=\omega\in \mu_n$. We separate $S_{n,2}(q)$ into two sums -- one with the summands corresponding to $\ell$ divisible by $\mathrm{ord}(\omega)$, and another with the rest:
\begin{equation*}
S_{n,2}(q) = T_1(q) + T_2(q),
\end{equation*}
where
\begin{align*}
T_1(q) &= 2\sum_{k,\ell, \, \mathrm{ord}(\omega) \mid \ell} {n \brack k}_q^2 q^{f(n,k,\ell)} q^2 {n \brack \ell}'_q {k \brack \ell}'_q {k+\ell \brack n}_q, \\
T_2(q) &= 2\sum_{k,\ell, \, \mathrm{ord}(\omega) \nmid \ell} {n \brack k}_q^2 q^{f(n,k,\ell)} q^2 {n \brack \ell}'_q {k \brack \ell}'_q {k+\ell \brack n}_q. 
\end{align*}
We use Lemma~\ref{lem:binomnplusk} and Corollary~\ref{cor:binomder}
to evaluate $T_1(\omega)$, a sum that is supported on $k$-s and $\ell$-s divisible by $\mathrm{ord}(\omega)$:
\begin{equation}\label{t1valz}
\begin{split}
T_1(\omega) &= 2\sum_{\mathrm{ord}(\omega) \mid k, \ell}  \binom{\frac{n}{\mathrm{ord}(\omega)}}{\frac{k}{\mathrm{ord}(\omega)}}^2 \binom{\frac{n}{\mathrm{ord}(\omega)}}{\frac{\ell}{\mathrm{ord}(\omega)}} \binom{\frac{k}{\mathrm{ord}(\omega)}}{\frac{\ell}{\mathrm{ord}(\omega)}} \binom{\frac{k+\ell}{\mathrm{ord}(\omega)}}{\frac{n}{\mathrm{ord}(\omega)}} \frac{\ell(n-\ell)}{2} \frac{\ell(k-\ell)}{2}\\
&=\mathrm{ord}(\omega)^4 a_{\frac{n}{\mathrm{ord}(\omega)},\frac{1}{2}(z^2(x-z)(y-z))}.
\end{split}
\end{equation}
We turn to $T_2(\omega)$. By Lemma~\ref{lem:binomnplusk} and Corollary~\ref{cor:binomder}, we may evaluate $T_2(\omega)$ as follows:
\begin{equation}\label{eq:t2omega1z}
T_2(\omega) = 2\sum_{\mathrm{ord}(\omega) \mid k, \, \mathrm{ord}(\omega) \nmid \ell}
\binom{\frac{n}{\mathrm{ord}(\omega)}}{ \frac{k}{\mathrm{ord}(\omega)}}^2  \frac{\omega^{f(n,k,\ell)-\ell^2+\ell}}{(\omega^{\ell}-1)^2} nk \binom{\frac{n}{\mathrm{ord}(\omega)}-1}{\lfloor \frac{\ell-1}{\mathrm{ord}(\omega)} \rfloor}
\binom{\frac{k}{\mathrm{ord}(\omega)}-1}{\lfloor \frac{\ell-1}{\mathrm{ord}(\omega)} \rfloor} \binom{\frac{k}{\mathrm{ord}(\omega)} + \lfloor \frac{\ell}{\mathrm{ord}(\omega)}\rfloor}{\frac{n}{\mathrm{ord}(\omega)}}.
\end{equation}
We write every $0 \le \ell \le n$ with $\mathrm{ord}(\omega) \nmid \ell$ as $\ell=\mathrm{ord}(\omega)j+i$ with $1 \le i \le \mathrm{ord}(\omega)-1$ and $0 \le j \le \frac{n}{\mathrm{ord}(\omega)}-1$. Now we have, for $\omega \in \mu_n$ and $\mathrm{ord}(\omega) \mid k$, the equalities $\omega^{\ell}=\omega^i$, $\omega^{f(n,k,\ell)}=\omega^{\ell^2} = \omega^{i^2}$, $\lfloor \frac{\ell}{\mathrm{ord}(\omega)} \rfloor = \lfloor \frac{\ell-1}{\mathrm{ord}(\omega)} \rfloor = j$. Letting $k'=\frac{k}{\mathrm{ord}(\omega)}$, we express \eqref{eq:t2omega1z} as
\begin{equation}\label{eq:t2omega2z}
T_2(\omega) =2\mathrm{ord}(\omega)n \sum_{k'=0}^{\frac{n}{\mathrm{ord}(\omega)}} \sum_{j=0}^{\frac{n}{\mathrm{ord}(\omega)}-1} \binom{\frac{n}{\mathrm{ord}(\omega)}}{k'}^2 k' \binom{\frac{n}{\mathrm{ord}(\omega)}-1}{j} \binom{k'-1}{j} \binom{k'+j}{\frac{n}{\mathrm{ord}(\omega)}} \sum_{i=1}^{\frac{n}{\mathrm{ord}(\omega)}-1} 
\frac{\omega^i}{(\omega^i-1)^2}.
\end{equation}
The relation $\binom{a-1}{b} = \binom{a}{b} \frac{a-b}{a}$ and Lemma~\ref{lem:rootsunitysum} allow us to simplify \eqref{eq:t2omega2z} as
\begin{equation}\label{eq:ressumroots}
T_2(\omega) = -\frac{\mathrm{ord}(\omega)^2}{6}(\mathrm{ord}(\omega)^2-1) a_{\frac{n}{\mathrm{ord}(\omega)}, (x-z)(y-z)}.
\end{equation}
From \eqref{t1valz} and \eqref{eq:ressumroots}, we obtain
\begin{equation}\label{eq:sn2evalz}
S_{n,2}(\omega)=\mathrm{ord}(\omega)^4
 a_{\frac{n}{\mathrm{ord}(\omega)},(\frac{z^2}{2}-\frac{1}{6})(x-z)(y-z)} + \mathrm{ord}(\omega)^2 a_{\frac{n}{\mathrm{ord}(\omega)}, \frac{(x-z)(y-z)}{6}}.
\end{equation}
We turn to evaluate $S_{n,3}(q)$ at $q=\omega\in \mu_n$. We separate $S_{n,3}(q)$ into two sums -- one with the summands corresponding to $k$ divisible by $\mathrm{ord}(\omega)$, and another with the rest:
\begin{equation*}
S_{n,3}(q) = U_1(q) + U_2(q),
\end{equation*}
where
\begin{align*}
U_1(q) &= 2\sum_{k,\ell, \, \mathrm{ord}(\omega) \mid k} (q{n \brack k}'_q)^2 q^{f(n,k,\ell)}  {n \brack \ell}_q {k \brack \ell}_q {k+\ell \brack n}_q,\\
U_2(q) &= 2\sum_{k,\ell, \, \mathrm{ord}(\omega) \nmid k} (q{n \brack k}'_q)^2 q^{f(n,k,\ell)} {n \brack \ell}_q {k \brack \ell}_q {k+\ell \brack n}_q. 
\end{align*}
We use Lemma~\ref{lem:binomnplusk} and Corollary~\ref{cor:binomder}
to evaluate $U_1(\omega)$, a sum that is supported on $k$-s and $\ell$-s divisible by $\mathrm{ord}(\omega)$:
\begin{equation}\label{u1val}
\begin{split}
U_1(\omega) &= 2\sum_{\mathrm{ord}(\omega) \mid k, \ell}  \binom{\frac{n}{\mathrm{ord}(\omega)}}{\frac{k}{\mathrm{ord}(\omega)}}^2 \binom{\frac{n}{\mathrm{ord}(\omega)}}{\frac{\ell}{\mathrm{ord}(\omega)}} \binom{\frac{k}{\mathrm{ord}(\omega)}}{\frac{\ell}{\mathrm{ord}(\omega)}} \binom{\frac{k+\ell}{\mathrm{ord}(\omega)}}{\frac{n}{\mathrm{ord}(\omega)}} (\frac{k(n-k)}{2})^2\\
&=\mathrm{ord}(\omega)^4 a_{\frac{n}{\mathrm{ord}(\omega)},\frac{1}{2}(y^2(x-y)^2)}.
\end{split}
\end{equation}
We turn to $U_2(\omega)$. By Lemma~\ref{lem:binomnplusk} and Corollary~\ref{cor:binomder}, we may evaluate $U_2(\omega)$ as follows:
\begin{equation}\label{eq:u2omega1}
U_2(\omega) = 2\sum_{\mathrm{ord}(\omega) \mid \ell, \, \mathrm{ord}(\omega) \nmid k}
(n\frac{\omega^{-\binom{k}{2}}}{\omega^k-1} \binom{\frac{n}{\mathrm{ord}(\omega)}-1}{\lfloor \frac{k-1}{\mathrm{ord}(\omega)} \rfloor})^2  \omega^{f(n,k,\ell)}  \binom{\frac{n}{\mathrm{ord}(\omega)}}{ \frac{\ell}{\mathrm{ord}(\omega)} }
\binom{\lfloor \frac{k}{\mathrm{ord}(\omega)}\rfloor}{\frac{\ell}{\mathrm{ord}(\omega)}} \binom{\lfloor \frac{k}{\mathrm{ord}(\omega)} \rfloor+ \frac{\ell}{\mathrm{ord}(\omega)}}{\frac{n}{\mathrm{ord}(\omega)}}.
\end{equation}
We write every $0 \le k \le n$ with $\mathrm{ord}(\omega) \nmid k$ as $k=\mathrm{ord}(\omega)j+i$ with $1 \le i \le \mathrm{ord}(\omega)-1$ and $0 \le j \le \frac{n}{\mathrm{ord}(\omega)}-1$. Now we have, for $\omega \in \mu_n$ and $\mathrm{ord}(\omega) \mid \ell$, the equalities $\omega^{k}=\omega^i$, $\omega^{f(n,k,\ell)}=\omega^{k^2} = \omega^{i^2}$, $\lfloor \frac{k}{\mathrm{ord}(\omega)} \rfloor = \lfloor \frac{k-1}{\mathrm{ord}(\omega)} \rfloor = j$. Letting $\ell'=\frac{\ell}{\mathrm{ord}(\omega)}$, we express \eqref{eq:u2omega1} as
\begin{equation}\label{eq:u2omega2}
U_2(\omega) =2n^2 \sum_{\ell'=0}^{\frac{n}{\mathrm{ord}(\omega)}} \sum_{j=0}^{\frac{n}{\mathrm{ord}(\omega)}-1} \binom{\frac{n}{\mathrm{ord}(\omega)}-1}{j}^2  \binom{\frac{n}{\mathrm{ord}(\omega)}}{\ell'} \binom{j}{\ell'} \binom{j+\ell'}{\frac{n}{\mathrm{ord}(\omega)}} \sum_{i=1}^{\frac{n}{\mathrm{ord}(\omega)}-1} 
\frac{\omega^i}{(\omega^i-1)^2}.
\end{equation}
The relation $\binom{a-1}{b} = \binom{a}{b} \frac{a-b}{a}$ and Lemma~\ref{lem:rootsunitysum} allow us to simplify \eqref{eq:u2omega2} as
\begin{equation}\label{eq:ressumroots2}
U_2(\omega) = -\frac{\mathrm{ord}(\omega)^2}{6}(\mathrm{ord}(\omega)^2-1) a_{\frac{n}{\mathrm{ord}(\omega)}, (x-y)^2}.
\end{equation}
From \eqref{u1val} and \eqref{eq:ressumroots2}, we obtain
\begin{equation}\label{eq:sn3eval}
S_{n,3}(\omega)=\mathrm{ord}(\omega)^4
 a_{\frac{n}{\mathrm{ord}(\omega)},\frac{1}{2}y^2(x-y)^2 - \frac{(x-y)^2}{6}} + \mathrm{ord}(\omega)^2 a_{\frac{n}{\mathrm{ord}(\omega)}, \frac{1}{6}(x-y)^2}.
\end{equation}
From \eqref{eq:sn1valz}, \eqref{eq:sn2evalz} and \eqref{eq:sn3eval}, we obtain \eqref{eq:an3z}, as needed. \qed 
\subsection{Alternative form}\label{sec:altern}
In Theorems~\ref{thm:superqbinom}--\ref{thm:zeta} we have calculated the first three $[n]_q$-digits of several sequences, which allowed us to obtain supercongruences modulo $n^3$. For $n=p$ a prime we describe below an alternative way to deduce the supercongruences, which seems more elegant, although it is not as general as Corollary~\ref{corexpansion}.
\begin{proposition}\label{prop:np}
Let $\{ a_n(q) \}_{n \ge 1} \subseteq \mathbb{Z}[q]$ be a sequence satisfying the $q$-Gauss congruences. Suppose that for any $n \ge 1$ and $\omega \in \mu_n$, we have
\begin{equation}\label{eq:CondAlternative}
\omega a_n'(\omega) = \mathrm{ord}(\omega)^2 a_{\frac{n}{\mathrm{ord}(\omega)}}'(1).
\end{equation}
Moreover, suppose that there are sequences $b_n$, $c_n$ such that for any $n \ge 1$ and $\omega \in \mu_n$, we have
\begin{equation}\label{eq:CondAlternative2}
\omega^2 a_n''(\omega) = \mathrm{ord}(\omega)^4 b_{\frac{n}{\mathrm{ord}(\omega)}} + \mathrm{ord}(\omega)^2 c_{\frac{n}{\mathrm{ord}(\omega)}}.
\end{equation}
Then for any $m,n \ge 1$ with $(m,6)=1$ we have
\begin{equation}\label{eq:residuep3}
 a_{nm}(q) - a_{n}(q^{m^2}) \equiv -(q^m-1)^2  \frac{m^2-1}{2} (c_n+a_n'(1)) \bmod \Phi_m(q)^3.
\end{equation}
In particular for $m=p \ge 5$ a prime we have
\begin{equation}\label{eq:supernew}
a_{np}(1) \equiv a_n(1) \bmod p^3.
\end{equation}
\end{proposition}
\begin{proof}
Using \eqref{eq:CondAlternative2} with $\omega=1$, it follows that $b_n+c_n$ is an integer. Using \eqref{eq:CondAlternative2} with $\omega=-1$ and $2n$ in place of $n$, we find that $16b_n + 4c_n$ is an integer. These two integrality conditions imply that $b_n,c_n$ are rational numbers with denominator dividing $12$. In particular, the right-hand side of \eqref{eq:residuep3} has integer coefficients since $24 \mid m^2-1$ if $(m,6)=1$.

Let $g_1(q) = a_{nm}(q) - a_n(q^{m^2})$ and $g_2(q)=-(q^m-1)(m^2-1)(c_n + a'_n(1))/2$. As $a_n(q)$ satisfies the $q$-Gauss congruences, Corollary~\ref{cor:sp} implies that $g_1(\omega)$ vanishes on the zeros of $\Phi_m(q)$, and \eqref{eq:CondAlternative} ensures that $g_1'(\omega)$ vanishes on these zeros again. We also have that $g_2$ vanishes twice on the zeros of $\Phi_m(q)$ by construction. Using \eqref{eq:CondAlternative} and \eqref{eq:CondAlternative2}, we have
\begin{equation}
\begin{split}
g_1''(\omega) &= a''_{nm}(\omega) - (m^2 \omega^{m^2-1})^2 a_n''(\omega^{m^2}) - (m^2(m^2-1)\omega^{m^2-2})a_n'(\omega^{m^2}) \\
&= \frac{m^4 b_n + m^2 c_n}{\omega^2} - \frac{m^4(b_n+c_n)}{\omega^2} - \frac{m^2(m^2-1)a_n'(1)}{\omega^2}=\frac{-m^2(m^2-1)(c_n+a_n'(1))}{\omega^2}
\end{split}
\end{equation}
for any $\omega$ which is a primitive root of unity of order $m$, and similarly
\begin{equation}
g_2''(\omega) =\frac{-m^2(m^2-1)(c_n+a_n'(1))}{\omega^2}.
\end{equation}
Thus, $(g_1-g_2)''(\omega)=0$ for every zero of $\Phi_m$, which establishes \eqref{eq:residuep3}. To establish \eqref{eq:supernew}, we apply \eqref{eq:residuep3} with $m=p$ and plug $q=1$.
\end{proof}
For $a_n(q) = {an \brack bn}_q$, Corollary~\ref{cor:binomder} shows that the conditions of Proposition~\ref{prop:np} hold with $a_n'(1) = \binom{an}{bn}\frac{b(a-b)}{2}n^2$, $b_n = \binom{an}{bn}\frac{b(a-b)n^2}{4} (b(a-b)n^2 + \frac{an}{3})$, $c_n=-\frac{5}{12} \binom{an}{bn}b(a-b)n^2$. Applying the proposition with $n=1$, we obtain \eqref{straubres2}. 

For $a_n(q) = \sum_{k=0}^{n} {n \brack k}_q^2 {n+k \brack k}_q^2 q^{f(n,k)}$ with $f(x,y)$ as in Theorem~\ref{thm:aprey}, then the proof of Theorem~\ref{thm:aprey} shows that the conditions of Proposition~\ref{prop:np} hold with $a_n'(1) = a_{n,2xy-y^2+f(x,y)}$, $b_n=a_{n,(2xy-y^2+f(x,y))^2+\frac{x^2y}{3}-\frac{(x-y)^2}{6}}$ and $c_n=a_{n,\frac{x^2}{6}-2xy+y^2-f(x,y)}$. Applying the proposition we obtain \eqref{eq:aperyqelegant}.

Finally, for $a_n(q) = \sum_{k,\ell} {n \brack k}_q^2 {n \brack \ell}_q {k \brack \ell}_q {k+\ell \brack n}_q q^{f(n,k,\ell)}$ with $f(x,y,z)$ as in Theorem~\ref{thm:zeta}, the proof of Theorem~\ref{thm:zeta} shows that the conditions of Proposition~\ref{prop:np} hold with $a'_n(1) = a_{n,xz-y^2-z^2+ \frac{3xy+yz-x^2}{2}+f(x,y,z)}$, $b_n = a_{n,Q_1}$ and $c_n = a_{n,Q_2}$, where $Q_1$, $Q_2$ are given in Theorem~\ref{thm:zeta}. Applying the proposition we obtain \eqref{eq:zetaqelegant}.
\section*{Acknowledgments}
The author would like to thank Ira Gessel and Richard Stanley for the references in \S\ref{seccritgauss}, and the referee for valuable suggestions. He is also grateful to Wadim Zudilin for discussions on $q$-congruences and useful suggestions regarding the paper. He is especially thankful to Armin Straub for careful suggestions that improved an earlier draft.

\bibliographystyle{alpha}
\bibliography{references}

\newcommand{\etalchar}[1]{$^{#1}$}
\begin{thebibliography}{ABDJ17}

\bibitem[AB87]{andrews1987}
George~E. Andrews and R.~J. Baxter.
\newblock Lattice gas generalization of the hard hexagon model. {III}.
  {$q$}-trinomial coefficients.
\newblock {\em J. Statist. Phys.}, 47(3-4):297--330, 1987.

\bibitem[ABDJ17]{adamczewski2017}
B.~Adamczewski, J.~P. Bell, \'{E}. Delaygue, and F.~Jouhet.
\newblock Congruences modulo cyclotomic polynomials and algebraic independence
  for {$q$}-series.
\newblock {\em S\'{e}m. Lothar. Combin.}, 78B:Art. 54, 12, 2017.

\bibitem[And90a]{andrews1990}
George~E. Andrews.
\newblock Euler's ``exemplum memorabile inductionis fallacis'' and
  {$q$}-trinomial coefficients.
\newblock {\em J. Amer. Math. Soc.}, 3(3):653--669, 1990.

\bibitem[And90b]{andrews1990q}
George~E. Andrews.
\newblock {$q$}-trinomial coefficients and {R}ogers-{R}amanujan type
  identities.
\newblock In {\em Analytic number theory ({A}llerton {P}ark, {IL}, 1989)},
  volume~85 of {\em Progr. Math.}, pages 1--11. Birkh\"auser Boston, Boston,
  MA, 1990.

\bibitem[And94]{andrews1994}
George~E. Andrews.
\newblock Schur's theorem, {C}apparelli's conjecture and {$q$}-trinomial
  coefficients.
\newblock In {\em The {R}ademacher legacy to mathematics ({U}niversity {P}ark,
  {PA}, 1992)}, volume 166 of {\em Contemp. Math.}, pages 141--154. Amer. Math.
  Soc., Providence, RI, 1994.

\bibitem[And99]{andrews1999}
George~E. Andrews.
\newblock {$q$}-analogs of the binomial coefficient congruences of {B}abbage,
  {W}olstenholme and {G}laisher.
\newblock {\em Discrete Math.}, 204(1-3):15--25, 1999.

\bibitem[And04]{andrews2004}
George~E. Andrews.
\newblock Fibonacci numbers and the {R}ogers-{R}amanujan identities.
\newblock {\em Fibonacci Quart.}, 42(1):3--19, 2004.

\bibitem[Ap{\'e}79]{apery1979}
Roger Ap{\'e}ry.
\newblock {Irrationalit\'e de $\zeta(2)$ et $\zeta(3)$.}
\newblock {\em {Ast\'erisque}}, 61:11--13, 1979.

\bibitem[AT16]{amdeberhan2016}
Tewodros Amdeberhan and Roberto Tauraso.
\newblock Supercongruences for the {A}lmkvist-{Z}udilin numbers.
\newblock {\em Acta Arith.}, 173(3):255--268, 2016.

\bibitem[AvSZ11]{almkvist2011}
Gert Almkvist, Duco van Straten, and Wadim Zudilin.
\newblock Generalizations of {C}lausen's formula and algebraic transformations
  of {C}alabi-{Y}au differential equations.
\newblock {\em Proc. Edinb. Math. Soc. (2)}, 54(2):273--295, 2011.

\bibitem[AZ06]{almkvist2006}
Gert Almkvist and Wadim Zudilin.
\newblock Differential equations, mirror maps and zeta values.
\newblock In {\em Mirror symmetry. {V}}, volume~38 of {\em AMS/IP Stud. Adv.
  Math.}, pages 481--515. Amer. Math. Soc., Providence, RI, 2006.

\bibitem[Beu85]{beukers1985}
F.~Beukers.
\newblock Some congruences for the {A}p\'ery numbers.
\newblock {\em J. Number Theory}, 21(2):141--155, 1985.

\bibitem[Beu87]{beukers1987}
F.~Beukers.
\newblock Another congruence for the {A}p\'ery numbers.
\newblock {\em J. Number Theory}, 25(2):201--210, 1987.

\bibitem[BSF{\etalchar{+}}52]{brun1949}
Viggo Brun, J.~O. Stubban, J.~E. Fjeldstad, R.~Tambs~Lyche, K.~E. Aubert,
  W.~Ljunggren, and E.~Jacobsthal.
\newblock On the divisibility of the difference between two binomial
  coefficients.
\newblock In {\em Den 11te {S}kandinaviske {M}atematikerkongress, {T}rondheim,
  1949}, pages 42--54. Johan Grundt Tanums Forlag, Oslo, 1952.

\bibitem[Car74]{carlitz1974}
L.~Carlitz.
\newblock Fibonacci notes. {III}. {$q$}-{F}ibonacci numbers.
\newblock {\em Fibonacci Quart.}, 12:317--322, 1974.

\bibitem[Car75]{carlitz1975}
L.~Carlitz.
\newblock Fibonacci notes. {IV}. {$q$}-{F}ibonacci polynomials.
\newblock {\em Fibonacci Quart.}, 13:97--102, 1975.

\bibitem[CCC80]{chowla1980}
S.~Chowla, J.~Cowles, and M.~Cowles.
\newblock Congruence properties of {A}p\'ery numbers.
\newblock {\em J. Number Theory}, 12(2):188--190, 1980.

\bibitem[CGP01]{cai2001}
Tianxin Cai and Gilberto Garc\'{i}a-Pulgar\'{i}n.
\newblock Two {W}olstenholme's type theorems on {$q$}-binomial coefficients.
\newblock {\em Rev. Colombiana Mat.}, 35(2):61--65, 2001.

\bibitem[Cha11]{chan2011}
Hei-Chi Chan.
\newblock {\em An invitation to {$q$}-series}.
\newblock World Scientific Publishing Co. Pte. Ltd., Hackensack, NJ, 2011.
\newblock From {J}acobi's triple product identity to {R}amanujan's ``most
  beautiful identity''.

\bibitem[Cig03]{cigler2003}
Johann Cigler.
\newblock {$q$}-{F}ibonacci polynomials.
\newblock {\em Fibonacci Quart.}, 41(1):31--40, 2003.

\bibitem[Cig04]{cigler2004}
Johann Cigler.
\newblock {$q$}-{F}ibonacci polynomials and the {R}ogers-{R}amanujan
  identities.
\newblock {\em Ann. Comb.}, 8(3):269--285, 2004.

\bibitem[Cig16]{cigler2016}
Johann Cigler.
\newblock Some divisibility properties of q-{F}ibonacci numbers.
\newblock {\em arXiv preprint arXiv:1604.07977}, 2016.

\bibitem[Cla95]{clark1995}
W.~Edwin Clark.
\newblock {$q$}-analogue of a binomial coefficient congruence.
\newblock {\em Internat. J. Math. Math. Sci.}, 18(1):197--200, 1995.

\bibitem[Coh04]{cohn2004}
Henry Cohn.
\newblock Projective geometry over {$F_1$} and the {G}aussian binomial
  coefficients.
\newblock {\em Amer. Math. Monthly}, 111(6):487--495, 2004.

\bibitem[Coo12]{cooper2012}
Shaun Cooper.
\newblock Sporadic sequences, modular forms and new series for {$1/\pi$}.
\newblock {\em Ramanujan J.}, 29(1-3):163--183, 2012.

\bibitem[Cos88]{coster1988}
Matthijs~Johannes Coster.
\newblock {\em Supercongruences}.
\newblock PhD thesis, University of Leiden, 1988.

\bibitem[DHL03]{du2003}
Bau-Sen Du, Sen-Shan Huang, and Ming-Chia Li.
\newblock Generalized {F}ermat, double {F}ermat and {N}ewton sequences.
\newblock {\em J. Number Theory}, 98(1):172--183, 2003.

\bibitem[DS06]{deutsch2006}
Emeric Deutsch and Bruce~E. Sagan.
\newblock Congruences for {C}atalan and {M}otzkin numbers and related
  sequences.
\newblock {\em J. Number Theory}, 117(1):191--215, 2006.

\bibitem[Ges80]{gessel1980}
Ira Gessel.
\newblock A noncommutative generalization and {$q$}-analog of the {L}agrange
  inversion formula.
\newblock {\em Trans. Amer. Math. Soc.}, 257(2):455--482, 1980.

\bibitem[Ges82]{gessel1982}
Ira Gessel.
\newblock Some congruences for {A}p\'ery numbers.
\newblock {\em J. Number Theory}, 14(3):362--368, 1982.

\bibitem[Gil89]{gillespie1989}
Frank~S. Gillespie.
\newblock A generalization of {F}ermat's little theorem.
\newblock {\em Fibonacci Quart.}, 27(2):109--115, 1989.

\bibitem[GZ18]{guo2018}
Victor J.~W. Guo and Wadim Zudilin.
\newblock A $q$-microscope for supercongruences.
\newblock {\em arXiv preprint arXiv:1803.01830}, 2018.

\bibitem[IR90]{ireland2013}
Kenneth Ireland and Michael Rosen.
\newblock {\em A classical introduction to modern number theory}, volume~84 of
  {\em Graduate Texts in Mathematics}.
\newblock Springer-Verlag, New York, second edition, 1990.

\bibitem[JM06]{jezierski2006}
Jerzy Jezierski and Wac\l{a}w Marzantowicz.
\newblock {\em Homotopy methods in topological fixed and periodic points
  theory}, volume~3 of {\em Topological Fixed Point Theory and Its
  Applications}.
\newblock Springer, Dordrecht, 2006.

\bibitem[KRZ06]{krattenthaler2006}
C.~Krattenthaler, T.~Rivoal, and W.~Zudilin.
\newblock S\'eries hyperg\'eom\'etriques basiques, {$q$}-analogues des valeurs
  de la fonction z\^eta et s\'eries d'{E}isenstein.
\newblock {\em J. Inst. Math. Jussieu}, 5(1):53--79, 2006.

\bibitem[Mim83]{mimura1983}
Yoshio Mimura.
\newblock Congruence properties of {A}p\'ery numbers.
\newblock {\em J. Number Theory}, 16(1):138--146, 1983.

\bibitem[Min14]{minton2014}
Gregory~T. Minton.
\newblock Linear recurrence sequences satisfying congruence conditions.
\newblock {\em Proc. Amer. Math. Soc.}, 142(7):2337--2352, 2014.

\bibitem[MS16]{malik2016}
Amita Malik and Armin Straub.
\newblock Divisibility properties of sporadic {A}p\'ery-like numbers.
\newblock {\em Res. Number Theory}, 2:Art. 5, 26, 2016.

\bibitem[OS13]{osburn2013}
Robert Osburn and Brundaban Sahu.
\newblock A supercongruence for generalized {D}omb numbers.
\newblock {\em Funct. Approx. Comment. Math.}, 48(part 1):29--36, 2013.

\bibitem[OSS16]{osburn2016}
Robert Osburn, Brundaban Sahu, and Armin Straub.
\newblock Supercongruences for sporadic sequences.
\newblock {\em Proc. Edinb. Math. Soc. (2)}, 59(2):503--518, 2016.

\bibitem[Pan06]{pan2006}
Hao Pan.
\newblock Arithmetic properties of {$q$}-{F}ibonacci numbers and {$q$}-{P}ell
  numbers.
\newblock {\em Discrete Math.}, 306(17):2118--2127, 2006.

\bibitem[Pan08]{pan2008}
Hao Pan.
\newblock $q$-analogue of {G}auss' divisibility theorem.
\newblock {\em arXiv preprint arXiv:0804.0834}, 2008.

\bibitem[Pan13]{pan2013}
Hao Pan.
\newblock Congruences for {$q$}-{L}ucas numbers.
\newblock {\em Electron. J. Combin.}, 20(2):Paper 29, 8, 2013.

\bibitem[Rob00]{robert2000}
Alain~M. Robert.
\newblock {\em A course in {$p$}-adic analysis}, volume 198 of {\em Graduate
  Texts in Mathematics}.
\newblock Springer-Verlag, New York, 2000.

\bibitem[RSW04]{reiner2004}
V.~Reiner, D.~Stanton, and D.~White.
\newblock The cyclic sieving phenomenon.
\newblock {\em J. Combin. Theory Ser. A}, 108(1):17--50, 2004.

\bibitem[Sag11]{sagan2011}
Bruce~E. Sagan.
\newblock The cyclic sieving phenomenon: a survey.
\newblock In {\em Surveys in combinatorics 2011}, volume 392 of {\em London
  Math. Soc. Lecture Note Ser.}, pages 183--233. Cambridge Univ. Press,
  Cambridge, 2011.

\bibitem[Sch17]{schur1917}
I.~Schur.
\newblock {Ein Beitrag zur additiven Zahlentheorie und zur Theorie der
  Kettenbr\"uche.}
\newblock {\em {Berl. Ber.}}, 1917:302--321, 1917.

\bibitem[Sla08]{slavin2008}
Keith~R. Slavin.
\newblock {$q$}-binomials and the greatest common divisor.
\newblock {\em Integers}, 8:A05, 10, 2008.

\bibitem[SP07]{shi2007}
Ling-Ling Shi and Hao Pan.
\newblock A {$q$}-analogue of {W}olstenholme's harmonic series congruence.
\newblock {\em Amer. Math. Monthly}, 114(6):529--531, 2007.

\bibitem[Sta97]{stanley1997}
Richard~P. Stanley.
\newblock {\em Enumerative combinatorics. {V}ol. 1}, volume~49 of {\em
  Cambridge Studies in Advanced Mathematics}.
\newblock Cambridge University Press, Cambridge, 1997.
\newblock With a foreword by Gian-Carlo Rota, Corrected reprint of the 1986
  original.

\bibitem[Sta99]{stanley1999}
Richard~P. Stanley.
\newblock {\em Enumerative combinatorics. {V}ol. 2}, volume~62 of {\em
  Cambridge Studies in Advanced Mathematics}.
\newblock Cambridge University Press, Cambridge, 1999.
\newblock With a foreword by Gian-Carlo Rota and appendix 1 by Sergey Fomin.

\bibitem[Ste17]{steinlein2017}
Heinrich Steinlein.
\newblock Fermat's {L}ittle {T}heorem and {G}auss {C}ongruence: {M}atrix
  {V}ersions and {C}yclic {P}ermutations.
\newblock {\em Amer. Math. Monthly}, 124(6):548--553, 2017.

\bibitem[Str11]{straub2011}
Armin Straub.
\newblock {A $q$-analog of {L}junggren's binomial congruence.}
\newblock In {\em {Proceedings of the 23rd international conference on formal
  power series and algebraic combinatorics, FPSAC 2011, Reykjavik, Iceland,
  June 13--17, 2011}}, pages 897--902. Nancy: The Association. Discrete
  Mathematics \& Theoretical Computer Science (DMTCS), 2011.

\bibitem[Str19]{straub2019}
Armin Straub.
\newblock Supercongruences for polynomial analogs of the {A}p\'ery numbers.
\newblock {\em Proceedings of the American Mathematical Society},
  147:1023--1036, 2019.

\bibitem[War01]{warnaar2001}
S.~Ole Warnaar.
\newblock Refined {$q$}-trinomial coefficients and character identities.
\newblock In {\em Proceedings of the {B}axter {R}evolution in {M}athematical
  {P}hysics ({C}anberra, 2000)}, volume 102, pages 1065--1081, 2001.

\bibitem[War03]{warnaar2003}
S.~Ole Warnaar.
\newblock The generalized {B}orwein conjecture. {II}. {R}efined {$q$}-trinomial
  coefficients.
\newblock {\em Discrete Math.}, 272(2-3):215--258, 2003.

\bibitem[Zag09]{zagier2009}
Don Zagier.
\newblock Integral solutions of {A}p\'ery-like recurrence equations.
\newblock In {\em Groups and symmetries}, volume~47 of {\em CRM Proc. Lecture
  Notes}, pages 349--366. Amer. Math. Soc., Providence, RI, 2009.

\bibitem[Zar08]{zarelua2008}
A.~V. Zarelua.
\newblock On congruences for the traces of powers of some matrices.
\newblock {\em Proceedings of the Steklov Institute of Mathematics},
  263(1):78--98, 2008.

\bibitem[Zhe11]{zheng2011}
De-Yin Zheng.
\newblock An algebraic identity on {$q$}-{A}p\'ery numbers.
\newblock {\em Discrete Math.}, 311(23-24):2708--2710, 2011.

\bibitem[Zud19]{zudilin2019}
Wadim Zudilin.
\newblock Congruences for $q$-binomial coefficients.
\newblock {\em arXiv preprint arXiv:1901.07843}, 2019.

\end{thebibliography}

Raymond and Beverly Sackler School of Mathematical Sciences, Tel Aviv University, Tel Aviv 6997801, Israel.
E-mail address: ofir.goro@gmail.com

\end{document}